%% file: main.tex
\documentclass[a4paper,11pt]{article}
\usepackage{geometry}

\usepackage{graphicx}
\usepackage{hyperref}
\usepackage{amsmath,amssymb,amsthm,mathrsfs}
\usepackage{float}
\usepackage{placeins}
\usepackage{yhmath}
\usepackage{tikz}
\usepackage{caption}
\usepackage{subcaption}
\usepackage{blkarray}
\usepackage{booktabs}   
\usepackage{multirow}
\usepackage{xcolor}
\usepackage{colortbl}

\numberwithin{equation}{section}

\input{tex_src/mynewcommands}
\usetikzlibrary{shapes.geometric}

\newtheorem{theorem}{Theorem}
\newtheorem{lemma}{Lemma}
\newtheorem{corollary}{Corollary}

\theoremstyle{definition}
\newtheorem{definition}{Definition}

\theoremstyle{remark}
\newtheorem*{remark*}{Remark}

\title{Spectral Reality for Certain Fourth-Order Nonsymmetric Finite-Difference Laplacians}
\author{Yizhe Feng\(^{1}\), Weiguo Gao\(^{1,2,3}\), and Meiyue Shao\(^{2,3}\)\cr \small
24110180014@m.fudan.edu.cn, \{wggao,myshao\}@fudan.edu.cn}
\date{
\(^1\)School of Mathematical Sciences, Fudan University, Shanghai 200433, China\\\(^2\)School of Data Science, Fudan University, Shanghai 200433, China
\\\(^3\)Shanghai Key Laboratory of Contemporary Applied Mathematics, Shanghai 200433, China}

\begin{document}
\maketitle
\input{tex_src/abstract}
\input{tex_src/introduction}
\input{tex_src/preliminary}
\input{tex_src/spec_analysis_BH}
\input{tex_src/spec_analysis_third_order_boundary}
\input{tex_src/conclusion}
\input{tex_src/references}
\input{tex_src/appendix}
\end{document}

%% file: tex_src/mynewcommands.tex
\newcommand*{\trans}{^\top}

\newcommand{\diag}{\mathrm{diag}}
\newcommand{\tridiag}{\operatorname{Tridiag}}

\newcommand{\fupp}{f_{\mathrm{U}}}
\newcommand{\flow}{f_{\mathrm{L}}}
\newcommand{\gupp}{g_{\mathrm{U}}}
\newcommand{\glow}{g_{\mathrm{L}}}
\newcommand{\mupp}{m_{\mathrm{U}}}
\newcommand{\mlow}{m_{\mathrm{L}}}
\newcommand{\nupp}{n_{\mathrm{U}}}
\newcommand{\nlow}{n_{\mathrm{L}}}
\newcommand{\vupp}{v_{\mathrm{U}}}
\newcommand{\vlow}{v_{\mathrm{L}}}
\newcommand{\wupp}{w_{\mathrm{U}}}
\newcommand{\wlow}{w_{\mathrm{L}}}
\newcommand{\xupp}{X_{\mathrm{U}}}
\newcommand{\xlow}{X_{\mathrm{L}}}
\newcommand{\yupp}{Y_{\mathrm{U}}}
\newcommand{\ylow}{Y_{\mathrm{L}}}
\newcommand{\zupp}{Z_{\mathrm{U}}}
\newcommand{\zlow}{Z_{\mathrm{L}}}
\newcommand{\aupp}{\alpha}
\newcommand{\alow}{\beta}
\newcommand{\R}{\mathbb{R}}

\newcommand{\stepa}{\textbf{Step~A}}
\newcommand{\stepb}{\textbf{Step~B}}
\newcommand{\bdL}{L}
\newcommand{\bdA}{A}
\newcommand{\bdC}{C}
\newcommand{\bdF}{F}
\newcommand{\bdT}{T}
\newcommand{\bdP}{P}
\newcommand{\bdS}{S}
\newcommand{\bdI}{I}
\newcommand{\bdR}{R}
\newcommand{\bdU}{U}
\newcommand{\bdQ}{Q}

\newcommand{\bdM}{M}
\newcommand{\bdb}{b}
\newcommand{\bde}{e}

\newcommand{\bdx}{x}
\newcommand{\bdy}{y}
\newcommand{\bdz}{z}
\newcommand{\bdzero}{0}
\newcommand{\bdlda}{\Lambda}

\newcommand{\sign}{\operatorname{sign}}
\newcommand{\keywords}[1]{\textbf{Keywords:} #1}
\newcommand{\AMS}[1]{\textbf{AMS subject classifications (2020).} #1}

%% file: tex_src/abstract.tex
\begin{abstract}
Finite-difference discretizations of Laplace operators yield discrete Laplacians whose spectral properties are closely tied to the stability, convergence, and physical fidelity of numerical solvers.
While symmetric discretizations are well understood, many high-order finite-difference schemes produce nonsymmetric matrices for which rigorous spectral analysises are overlooked.
Surprisingly, we prove that two fourth-order schemes with Dirichlet boundary conditions yield nonsymmetric discrete Laplacians whose spectra are real and strictly positive.
Computer-assisted computations show that this property in general does not hold for higher-order extensions of these schemes.


\AMS{65N06, 15A18, 35J05}
\end{abstract}

%% file: tex_src/introduction.tex
\section{Introduction}

The Laplace operator plays a fundamental role in scientific computing.
Two classical problems governed by this operator are the Poisson
equation and the Laplace eigenvalue problem. 
For more information, see \cite{larssonPartialDifferentialEquations2003} and the references therein.
In this work, we focus on the latter problem under homogeneous Dirichlet boundary conditions:
\begin{equation}
    \label{eq:Laplace_eigenvalue_problem}
    -\Delta u(x) = \lambda u(x), \quad x \in \Omega,
    \qquad
    u(x) = 0, \quad x \in \partial\Omega.
\end{equation}
Under these boundary conditions, a key feature is that the negative Laplace operator \(-\Delta\) is well known to be positive definite. 
Various numerical methods have been developed to approximate its
eigenvalues and eigenfunctions, including the method of particular
solutions~(MPS)~\cite{foxApproximationsBounds1967}, finite-element methods
(FEM)~\cite{boffiFiniteElementApproximation2010,
sunFiniteElementMethods2016} and finite-difference methods
(FDM)~\cite{kuttlerFiniteDifferenceApproximations1970,
levequeFiniteDifferenceMethods2007}.


Finite-difference methods, including their high-order variants, have a long history of use in solving the Laplace eigenvalue problem, owing to their simplicity and efficiency~\cite{kuttlerFourthorderFinitedifferenceApproximation1971,kuttlerFiniteDifferenceApproximations1970}. 
Their structured discretizations can also be combined with fast algorithms such as the fast Fourier transform~(FFT)~\cite{brighamFastFourierTransform1988} in some suitable settings. 
Unlike standard FEM discretizations, whose discrete matrices are symmetric positive definite, high-order finite-difference schemes do not need to preserve this structure. In fact, they can produce nonsymmetric matrices, particularly when one-sided stencils are used near the boundary.
Despite their widespread use, the eigenvalue properties of these nonsymmetric matrices are often overlooked.

In this work, we identify a surprising spectral phenomenon in two fourth-order finite-difference schemes.
The Bramble--Hubbard scheme and a one-sided scheme produce  nonsymmetric discrete Laplacians, and all of their eigenvalues are real and strictly positive. 
We prove this property for both schemes. 
However, this property generally does not hold in higher-order extensions of the one-sided scheme. For certain orders and matrix dimensions, symbolic Sturm sequence computations~\cite{basuAlgorithmsRealAlgebraic2006} demonstrate the existence of nonreal conjugate eigenvalue pairs.\footnote{For every characteristic polynomial \(p\) considered, we verify exactly that \(\gcd(p,p')=1\).
Hence, \(p\) is square-free, and the Sturm count of distinct real roots equals the number of real roots counted with algebraic multiplicity.}

The paper is organized as follows. 
Section~\ref{sec:preliminaries} introduces the Laplace eigenvalue problem under consideration and reviews several finite-difference schemes for its discretization, including two fourth-order schemes, the Bramble--Hubbard scheme and a one-sided scheme, as well as a class of higher-order one-sided schemes. 
Sections~\ref{sec:spec_analysis_BH} and~\ref{sec:spec_analysis_third_order} present a detailed spectral analysis of the discrete Laplacians arising from the Bramble--Hubbard and the one-sided scheme, respectively.
We also remark that selected higher-order one-sided schemes may generate nonreal spurious eigenvalues.
Finally, Section~\ref{sec:conclusion} concludes the paper.

%% file: tex_src/preliminary.tex
\section{Preliminaries}
\label{sec:preliminaries}
In this section, we briefly introduce the discrete Laplacians generated by different finite-difference schemes applied to the Laplace eigenvalue problem~\eqref{eq:Laplace_eigenvalue_problem}.
For simplicity of notation, we consider the one-dimensional domain \(\Omega = (0,1)\) with a uniform grid \(0 = x_0 < x_1 < \dots < x_{n + 1} = 1\), where \(x_i = i/(n + 1)\), for \(0 \leq i \leq n + 1\) and the uniform grid spacing \(h = 1/(n + 1)\).
For a higher-dimensional hyperrectangle domain \(\Omega\), the corresponding discrete Laplacian can be constructed directly using the Kronecker sum; therefore, conclusions established in the one-dimensional case can easily be extended to higher dimensions.

\subsection{Fourth-Order FDM}
\label{subsec:fourth_order_FDM}
In this subsection, we introduce the fourth-order finite-difference scheme considered in this work.
For interior grid points $x_i$ with $2 \leq i \leq n-1$, a standard fourth-order central finite-difference approximation is adopted, that is
\begin{equation}
\label{eq:fourth_order_FDM}
    -\Delta_h u(x_i) = \frac{1}{12h^2} 
    \bigl(  u(x_{i-2}) 
        - 16u(x_{i-1}) 
        + 30u(x_{i})
        - 16u(x_{i+1})
        +   u(x_{i+2})
    \bigr).
\end{equation}
For near-boundary points, i.e., $i = 1$ and $i = n$, lower-order one-sided difference formulas are used to maintain consistency with the Dirichlet boundary conditions.
Under this setting, the discrete Laplacian generated by the fourth-order finite-difference scheme takes the following general matrix form:
\begin{equation}
\label{eq:discret_laplacian_general_form}
    \bdL = \frac{1}{12h^2}
    \begin{bmatrix}
        b_0 & b_1 & b_2 & b_3 \\
        a_{-1} & a_0 & a_1 & a_2      \\
        a_{-2} & a_{-1} & a_0 & a_1 & a_2 \\
               & a_{-2} & a_{-1} & a_0 & a_1 & a_2 \\
               & & \ddots & \ddots & \ddots & \ddots & \ddots \\
               & & & a_{-2} & a_{-1} & a_0 & a_1 & a_2 \\
               & & & & a_{-2} & a_{-1} & a_0 & a_1 \\
               & & & & b_3 & b_2 & b_1 & b_0
    \end{bmatrix} \in \R^{n \times n},
\end{equation}
where the matrix dimension \(n\) satisfies \(n \ge 5\).
The coefficients \(a_i\) are fixed and derived from the fourth-order central finite-difference scheme \eqref{eq:fourth_order_FDM}:
\begin{equation*}
    \bigl[a_{-2},\, a_{-1},\, a_0,\, a_1,\, a_2\bigr] 
    = \bigl[1,\; -16,\; 30,\; -16,\; 1\bigr].
\end{equation*}
We denote \( \bdb = [b_0,\, b_1,\, b_2,\, b_3] \) as the row vector collecting the effective coefficients in the first row of the discrete Laplacian~\eqref{eq:discret_laplacian_general_form}.
Several fourth-order finite-difference schemes considered in this study are summarized below.

\begin{itemize}
    \item \textbf{Bramble--Hubbard Scheme~\cite{brambleFormulationFiniteDifference1962,brambleFourthorderFiniteDifference1963}:}  
    In this scheme, the standard second-order central difference is applied at the near boundary point. 
    Specifically, for the first interior grid point,
    \begin{equation*}
        -\Delta_h u(x_1) = \frac{1}{h^2}\bigl(-u(x_0) + 2u(x_1) - u(x_2)\bigr).
    \end{equation*}
    Consequently, the corresponding coefficient vector \(\bdb\) is given by
    \begin{equation}
    \label{eq:Bramble--Hubbard scheme}
        \bdb = \bigl[24,\, -12,\, 0,\, 0\bigr].
    \end{equation}

    \item \textbf{Fourth-Order One-Sided Scheme~\cite{zhaoSpuriousSolutionsHighorder2007}:}  
    In this case, the standard second-order stencil near the boundary is replaced by a third-order finite-difference approximation. 
    Specifically, for the first interior grid point,
    \begin{equation*}
        -\Delta_h u(x_1) 
        = \frac{1}{12h^2}\bigl(-11u(x_0) + 20u(x_1) - 6u(x_2) - 4u(x_3) + u(x_4)\bigr).
    \end{equation*}
    Consequently, the first few stencil coefficients are modified as
    \begin{equation}
    \label{eq:Third-order near-boundary-accurate scheme}
        \bdb = \bigl[20,\, -6,\, -4,\, 1\bigr].
    \end{equation}
\end{itemize}

Figure~\ref{fig:fourth_order_eigenvalue_errors} plots the root mean square~(RMS) relative
errors of the minimal \(50\) eigenvalues of each scheme, matched with the exact Dirichlet eigenvalues \(\lambda_k=(k\pi)^2\).
Both curves show the fourth-order convergence as \(h=1/(n+1) \to 0\).
\begin{figure}[!tb]
    \centering
    \includegraphics[width=0.8\linewidth]{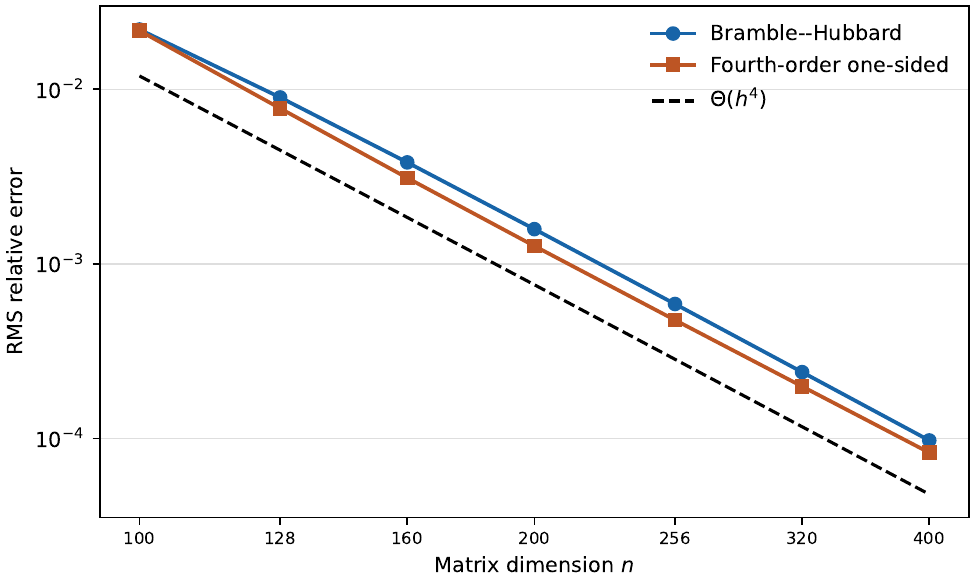}
    \caption{Eigenvalue errors for the Bramble--Hubbard and fourth-order
    one-sided schemes. The black dashed line indicates \(\Theta(h^4)\) scaling.}
    \label{fig:fourth_order_eigenvalue_errors}
\end{figure}


In this work, we establish the following theorems, which constitute the main theoretical results of this work.
\begin{theorem}
\label{thm:BH_real_positive}
    For every integer \(n \ge 5\), all eigenvalues of the discrete Laplacian constructed by the Bramble--Hubbard finite-difference scheme~\eqref{eq:Bramble--Hubbard scheme} are real and positive.
\end{theorem}

\begin{theorem}
\label{thm:third_order_real_positive}
    For every integer \(n \ge 5\), all eigenvalues of the discrete Laplacian constructed by the fourth-order one-sided finite-difference scheme~\eqref{eq:Third-order near-boundary-accurate scheme} are real and positive.
\end{theorem}


\subsection{Higher-Order Finite-Difference Schemes}
In this subsection, we present higher-order one-sided finite-difference schemes considered in~\cite{zhaoSpuriousSolutionsHighorder2007}.
This family generalizes the fourth-order one-sided scheme introduced in Section~\ref{subsec:fourth_order_FDM}, which is recovered by setting \(\ell=2\).
The discrete approximation at any grid point \(x_i\) (an interior point or near-boundary point) is expressed as
\begin{equation}
\label{eq:high_order_main_difference_scheme}
    -\Delta_h u(x_i) = \sum_{j=\mathrm{Low}_i}^{\mathrm{Upp}_i} c_{i,\,j}(x_i)\, u(x_j),
\end{equation}
where \(\mathrm{Low}_i\) and \(\mathrm{Upp}_i\) denote the lower and upper stencil indices, respectively, defined by
\begin{equation}
    \mathrm{Low}_i = \max(\min(i - \ell,\, n - 2\ell + 1),\, 0), \quad
    \mathrm{Upp}_i = \min(\max(2\ell,\, i + \ell),\, n + 1).
\end{equation}
These definitions apply for \(1 \le i \le n\) and \(n \ge 2\ell+1\).\footnote{The assumption \(n \ge 2\ell+1\) ensures that the \(2\ell+1\)-point stencil is well defined at every grid point.}
The coefficients \(c_{i,\,j}(x)\) correspond to the second derivative of the Lagrange interpolation kernel, defined for \(1 \le i \le n\) and
\(\mathrm{Low}_i \le j \le \mathrm{Upp}_i\) by
\begin{equation*}
    c_{i,j}(x) 
    = -\frac{{\rm d}^2}{{\rm d}x^2}
    \Biggl(\,\raisebox{0.8ex}{$\displaystyle\prod_{\substack{\mathrm{Low}_i\, \le\, k\, \le\, \mathrm{Upp}_i\\ k \ne j}}$}
    \,\frac{x - x_k}{x_j - x_k}\Biggr).
\end{equation*}
Hence, the discrete Laplace operator at each grid point is approximated using \(2\ell + 1\) neighboring points.
Under homogeneous Dirichlet boundary conditions, \(u(x_0)=u(x_{n+1})=0\); hence, the contributions associated with the boundary nodes \(x_0\) and \(x_{n+1}\) vanish and the corresponding columns are omitted from the discrete matrix.
The general nonzero stencil pattern of the resulting discrete Laplacian is shown in Figure~\ref{fig:higher_order_discrete_lap_scheme}. 

\begin{figure}[H]
    \centering
    \includegraphics[width=0.5\linewidth]{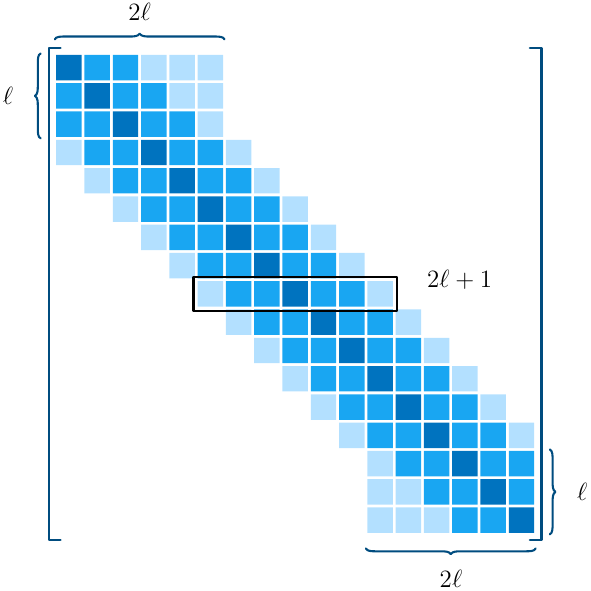}
    \caption{A schematic illustration of the general nonzero-stencil pattern for discrete Laplacians constructed by high-order one-sided finite difference schemes.}
    \label{fig:higher_order_discrete_lap_scheme}
\end{figure}

In this work, we further verify, for several orders of one-sided finite-difference schemes and selected matrix dimensions, that corresponding discrete Laplacians do not preserve the spectral property described in Section~\ref{subsec:fourth_order_FDM}. 
These results are confirmed through computer-assisted computations based on Sturm  sequence.

%% file: tex_src/spec_analysis_BH.tex
\section{Spectral Analysis of Bramble--Hubbard Scheme}
\label{sec:spec_analysis_BH}

In this section, we introduce our methodology for spectral analysis of the discrete Laplacian from Bramble--Hubbard scheme
\begin{equation}
\label{eq:discrete_Laplacian_matrix_BH}
L = 
\begin{bmatrix}
24 & -12 &  \\
-16 & 30 & -16 & 1 \\
1 & -16 & 30 & -16 & 1 \\
& \ddots & \ddots & \ddots & \ddots & \ddots & \\
& & 1 & -16 & 30 & -16 & 1 \\
& & & 1 & -16 & 30 & -16 \\
& & & & &-12 & 24
\end{bmatrix},
\end{equation}
where the leading coefficient $1/(12h^2)$ is omitted for notational simplicity.
This simplification does not affect the validity of our results.

Our approach consists of two main steps, referred to as \stepa\ and \stepb. 
\stepa\ computes the characteristic polynomial of the discrete Laplacian by employing the \emph{Wilkinson's similarity transformation} technique~\cite{wilkison}. 
\stepb\ then analyzes locations of  eigenvalues, equivalently, roots of the characteristic polynomial, by applying the intermediate-value theorem.

\subsection{\stepa: Computation for the Characteristic Polynomial}

A direct computation of the characteristic polynomial of \(\bdL\) is not straightforward and does not fully exploit its underlying structure.
To cover this issue, we we propose a novel method to derive the characteristic polynomial of the fourth-order discrete Laplacian.
A schematic overview of the procedure is provided in Figure~\ref{fig:schematic_overview_charac_poly_compute}.
\begin{figure}[!t]
    \centering
     \includegraphics[width=\linewidth]{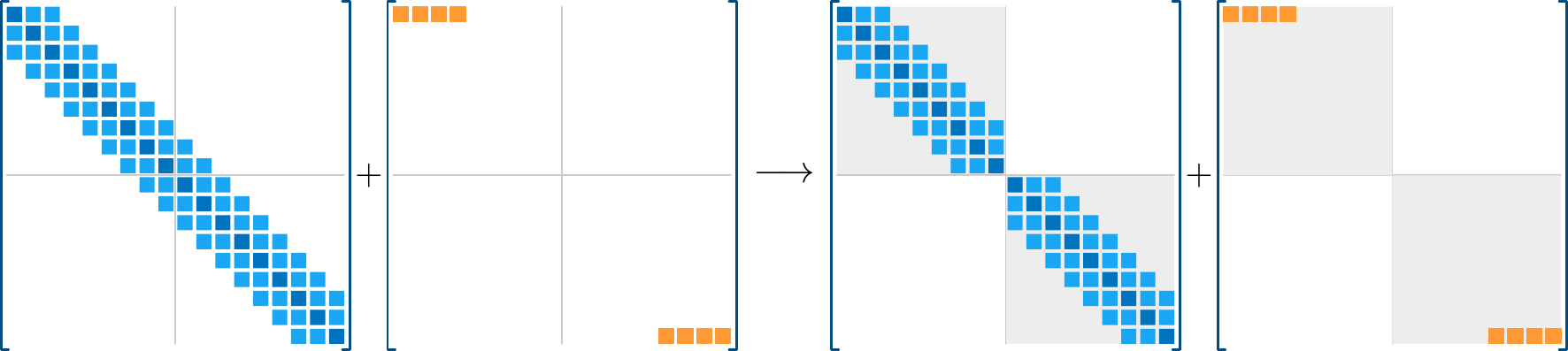}
    \caption{This is a schematic overview of the procedure for deriving the characteristic polynomial of the fourth-order discrete Laplacian.
    The discrete Laplacian is  decomposed into a symmetric pentadiagonal matrix $S$ together with a rank-two correction,  as shown  on the left.
    The symbol ``$\rightarrow$'' indicates the application of Wilkinson's similarity transformation to the decomposition, splitting the rank-two term into two rank-one terms, as shown on the right.}
    \label{fig:schematic_overview_charac_poly_compute}
\end{figure}

Motivated by perturbation theory, we introduce a symmetric pentadiagonal matrix 
\begin{equation}
\label{eq:S_test_matrix}
    \bdS =
    \begin{bmatrix}
        29 & -16 & 1 \\
        -16 & 30 & -16 & 1 \\
        1 & -16 & 30 & -16 & 1 \\
          & \ddots & \ddots & \ddots & \ddots & \ddots & \\
          &        & 1 & -16 & 30 & -16 & 1 \\
          &        &   & 1 & -16 & 30 & -16 \\
          &        &   &   & 1 & -16 & 29
    \end{bmatrix} \in \R^{n \times n}
\end{equation}
that serves as an approximation to the discrete Laplacian \(\bdL\), which is highly related to the two-step high-order compact (2SHOC) scheme \cite{caplanTwostepHighorderCompact2013}.
It admits the decomposition
\begin{equation*}
    \bdS = \bdT^2 - 16\bdT + 28\bdI,
\end{equation*}
where $ \bdT = \tridiag(1,0,1) \in \R^{n \times n}$ represents the tridiagonal Toeplitz matrix with the main diagonal elements set to $0$ and two sub-diagonal elements set to $1$.
Due to its well known eigendecomposition~\cite{noscheseTridiagonalToeplitzMatrices2013}, 
the approximate matrix \(\bdS\) can be orthogonally diagonalized as
\begin{equation}
\label{eq:S_spectrum_decomposition}
    \bdS = \bdQ \bdlda \bdQ\trans,
\end{equation}
where the orthogonal matrix \(\bdQ = [q_{i,j}]\) and the diagonal matrix $\bdlda= \diag(\mu_1, \dots, \mu_n)$ have entries, for \(i,j = 1,\dots,n\), given by
\begin{equation}
\label{eq:orthogonal_eigenmatrix_eigenvalue}
    q_{i,j} = \sqrt{\frac{2}{n+1}} \sin\frac{ij\pi}{n+1}, \quad \mu_i =  4\cos^2\!\frac{i\pi}{n+1} - 32\cos\frac{i\pi}{n+1} + 28.
\end{equation}
It is easy to verify that diagonal elements of $\bdlda$ are positive and arranged in ascending order. 

The discrete Laplacian \(\bdL\) can be regarded as a rank-two perturbation of \(\bdS\),
\begin{equation}
\label{eq:rank_two_BH}
    \bdL = \bdS + \bde_1 \bdz_1\trans + \bde_n \bdz_n\trans,
\end{equation}
where $\bde_i \in \R^n$ denotes the $i$-th column of the identity matrix and 
\begin{equation}
\label{eq:L-S_BH}
    \bdz_1 = \bigl[ -5, 4, -1, 0, \dots, 0\bigr]\trans, \quad
    \bdz_n = \bigl[0, \dots, 0, -1, 4, -5 \bigr]\trans.
\end{equation}

Consequently, directly deriving its characteristic polynomial from the eigendecomposition of \(\bdS\) remains challenging. 
To address this difficulty, we introduce the \emph{Wilkinson's similarity transformation} that exploits the centrosymmetric structure of \(\bdL\) and \(\bdS\), thereby further simplifying the spectral analysis.

\begin{definition}[Centrosymmetric matrix]
A matrix \(\bdM = [m_{i,j}] \in \R^{n \times n}\) is called \emph{centrosymmetric} if \(m_{i,j} = m_{n-i+1,\,n-j+1}\) for all \(1 \le i, \,j \le n\). 
Equivalently, it satisfies
\begin{equation*}
    \bdF \bdM \bdF = \bdM ,
\end{equation*}
where \(F = [e_n, e_{n-1}, \dots, e_2, e_1] \in \R^{n \times n}\) denotes the \emph{reverse-permutation} matrix.
\end{definition}

\begin{lemma}[\cite{weaverCentrosymmetricCrossSymmetricMatrices1985}]
\label{lemma:centro_block}
Let \(\bdM \in \R^{n \times n}\) be a centrosymmetric matrix.

\begin{enumerate}
\item If \(n = 2r\), then \(\bdM\) is orthogonally similar to the block form
\begin{equation}
\label{eq:lemma_centro_block_even}
\begin{aligned}
    & \bdP = \frac{1}{\sqrt{2}} 
    \begin{bmatrix}
        \bdI & \bdF \\
        -\bdF & \bdI
    \end{bmatrix},
    \qquad 
    \bdM = 
    \begin{bmatrix}
        \bdA & \bdF \bdC \bdF \\
        \bdC & \bdF \bdA \bdF
    \end{bmatrix}, \\
    & \hspace{1em} \bdP \bdM \bdP\trans = 
    \begin{bmatrix}
        \bdA + \bdF \bdC & \bdzero \\
        \bdzero & \bdF \bdA \bdF - \bdC \bdF
    \end{bmatrix},
\end{aligned}
\end{equation}
where \(\bdA, \bdC \in \R^{r \times r}\), and \(\bdF \in \R^{r \times r}\) is the reverse-permutation matrix.

\item If \(n = 2r + 1\), then \(M\) is orthogonally similar to the block form
\begin{equation}
\label{eq:lemma_centro_block_odd}
\begin{aligned}
    & \bdP = \frac{1}{\sqrt{2}} 
    \begin{bmatrix}
        \bdI & \bdzero & \bdF \\
        \bdzero & \sqrt{2} & \bdzero \\
        -\bdF & \bdzero & \bdI
    \end{bmatrix},
    \quad 
    \bdM = 
    \begin{bmatrix}
        \bdA & \bdx & \bdF \bdC \bdF \\
        \bdy\trans & q & \bdy\trans \bdF \\
        \bdC & \bdF \bdx & \bdF \bdA \bdF
    \end{bmatrix}, \\
    & \qquad 
    \bdP \bdM \bdP\trans = 
    \begin{bmatrix}
        \bdA + \bdF \bdC & \sqrt{2} \bdx & \bdzero \\
        \sqrt{2} \bdy\trans & q & \bdzero \\
        \bdzero & \bdzero & \bdF \bdA \bdF - \bdC \bdF
    \end{bmatrix},
\end{aligned}
\end{equation}
where \(\bdA, \bdC \in \R^{r \times r}\), \(\bdF \in \R^{r \times r}\) is the reverse-permutation matrix, and \(\bdx, \bdy \in \R^r\).
\end{enumerate}
\end{lemma}

Lemma~\ref{lemma:centro_block} reveals that, any centrosymmetric matrix admits an orthogonal similarity transformation that 
reduces it to a block diagonal form. 
It is observed that both $\bdL$ and $\bdS$ possess a centrosymmetric structure, 
and therefore their difference $\bdL - \bdS$ does as well. 
For \(n \geq 8\), they satisfy the block-diagonalization property described in 
Lemma~\ref{lemma:centro_block}:
\begin{align}
\label{eq:L^4_S_spectrum_decomposition_BH}
    & \bdP \bdL \bdP\trans = 
    \begin{bmatrix}
      \xupp & \bdzero \\
      \bdzero & \xlow
    \end{bmatrix},
    \quad
    \bdP \bdS \bdP\trans = 
    \begin{bmatrix}
      \yupp & \bdzero \\
      \bdzero & \ylow
    \end{bmatrix}, \\
\label{eq:L-S_spectrum_decomposition_BH}
    & \hspace{4em} 
    \bdP (\bdL - \bdS) \bdP\trans = 
    \begin{bmatrix}
      \zupp & \bdzero \\
      \bdzero & \zlow
    \end{bmatrix}.
\end{align}
For consistency of notation, we define \(\nupp = \lceil n / 2 \rceil\) and \(\nlow = \lfloor n / 2 \rfloor\), 
which represent the matrix dimensions of the upper block \(\xupp\) and the lower block \(\xlow\), respectively. 
From~\eqref{eq:rank_two_BH}, we know that the difference between \(\bdL\) and \(\bdS\) involves only a few entries located at the top-left and bottom-right corners. 
By combining the block forms in~\eqref{eq:lemma_centro_block_even} and~\eqref{eq:lemma_centro_block_odd}, 
we observe that the orthogonal similarity transformation leaves this difference invariant. 
Consequently, each block \(\zupp\) and \(\zlow\) can be expressed as
\begin{equation}
\label{eq:zupp_zlow_BH}
    \zupp = \vupp \,\wupp\trans, \quad 
    \zlow = \vlow \,\wlow\trans,
\end{equation}
where the vectors $\vupp,\, \wupp \in \R^{\nupp}$ 
are given by
\begin{equation}
\label{eq:f1fn_BH}
    \vupp = \bigl[1,\, 0,\, \dots,\, 0\bigr]\trans, 
    \quad 
    \wupp = \bigl[-5,\, 4,\, -1,\, 0,\, \dots,\, 0\bigr]\trans, 
\end{equation}
and the vectors $\vlow,\, \wlow \in \R^{\nlow}$ are given by
\begin{equation}
\label{eq:w1wn_BH}
    \vlow = \bigl[0,\, \dots,\, 0,\, 1\bigr]\trans,
    \quad
    \wlow = \bigl[0,\, \dots,\, 0,\, -1,\, 4,\, -5\bigr]\trans.
\end{equation}
Thus, these two matrices are of rank one.
Substituting~\eqref{eq:zupp_zlow_BH} into~\eqref{eq:L^4_S_spectrum_decomposition_BH} yields
\begin{equation}
\label{eq:two_rank_one_BH}
    \xupp = \yupp + \vupp \,\wupp\trans, 
    \quad 
    \xlow = \ylow + \vlow \,\wlow\trans.
\end{equation}

The rank-two correction \eqref{eq:rank_two_BH} can thus be decomposed into two rank-one corrections \eqref{eq:two_rank_one_BH}. 
This decomposition enables the characteristic polynomials of \(\xupp\) and \(\xlow\) to be computed separately 
by applying Sylvester’s determinant identity, 
together with the eigendecompositions of \(\yupp\) and \(\ylow\).

\begin{lemma}[Sylvester's determinant identity]
\label{lemma:sylvester}
    For any vector $v,\, w \in \R^n$, then \[\det(I + vw\trans) = 1 + v\trans w.\]
\end{lemma}

\begin{lemma}
\label{lemma:charac_poly_BH}
Let \(\bdL\) denote the discrete Laplacian \eqref{eq:discrete_Laplacian_matrix_BH} with matrix dimension $n \geq 5$, constructed by Bramble--Hubbard scheme, and let \(\xupp\) and \(\xlow\) denote the corresponding upper and lower diagonal blocks
in~\eqref{eq:L^4_S_spectrum_decomposition_BH}. 
Then the following results hold:
\begin{enumerate}
    \item The characteristic polynomial of \(\xupp\) can be expressed as 
    \begin{equation}
    \label{eq:fu_general_BH}
        \fupp(z) 
        = \Bigl(\,1 + \sum_{i=1}^{\nupp} 
        \frac{\aupp_i}{z - \mu_{2i-1}} \Bigr)
        \prod_{i=1}^{\nupp}(z - \mu_{2i-1}),
    \end{equation}
    where \(\mu_{2i-1}\) is defined in~\eqref{eq:orthogonal_eigenmatrix_eigenvalue},
    \(\theta = \pi/(n + 1)\), and the coefficients \(\aupp_i\)  are all positive and take the following forms: for \(1 \le i \le \nupp\),
    \begin{equation}
    \label{eq:bramble_hubbard_du}
        \aupp_i = 
        \frac{64}{n+1}
        \sin^4\!\frac{(2i-1)\pi}{2(n+1)}\,
        \sin^2\!\frac{(2i-1)\pi}{n+1}.
    \end{equation}
    
    \item The characteristic polynomial of \(\xlow\) can be expressed as 
    \begin{equation}
    \label{eq:fl_general_BH}
        \flow(z) 
        = \Bigl(\,1 + \sum_{j=1}^{\nlow} 
        \frac{\alow_j}{z - \mu_{2j}} \Bigr)
        \prod_{j=1}^{\nlow}(z - \mu_{2j}),
    \end{equation}
    where \(\mu_{2j}\) is defined in~\eqref{eq:orthogonal_eigenmatrix_eigenvalue},
    \(\theta = \pi/(n + 1)\), and the coefficients \(\alow_j\) are all positive and take the following forms: for \(1 \le j \le \nlow\),
    \begin{equation}
    \label{eq:bramble_hubbard_dl}
        \alow_j = 
        \frac{64}{n+1}
        \sin^4\!\frac{j\pi}{n+1}\,
        \sin^2\!\frac{2j\pi}{n+1}.
    \end{equation}

    \item The characteristic polynomial of \(\bdL\) can be written as 
    \begin{equation*}
        f_{\Delta}(z) = \fupp(z)\,\flow(z).
    \end{equation*}
\end{enumerate}

\end{lemma}

\begin{proof}
For \(n \in \{5,6,7\}\), the stated characteristic-polynomial identities are verified directly by exact computation.
We therefore assume \(n \ge 8\) for the remainder of the proof.

We only provide the proof of part~1, as the proof of part~2 is analogous, and part~3 follows directly from~\eqref{eq:L^4_S_spectrum_decomposition_BH}.
From \eqref{eq:S_spectrum_decomposition} and \eqref{eq:L^4_S_spectrum_decomposition_BH}, the matrix admits the following factorization:
\begin{equation}
\label{eq:yu_yl_BH}
    \begin{bmatrix}
        \yupp & 0 \\
        0 & \ylow 
    \end{bmatrix} = 
    (\bdP \bdQ) \bdlda (\bdP \bdQ)\trans
    = (\bdP \bdQ \bdR) (\bdR\trans \bdlda \bdR)  (\bdP \bdQ \bdR)\trans,
\end{equation}
where 
\begin{equation}
\label{eq:permutation_matrix_R}
    \bdR = [e_1, e_3, \dots, e_{2\nupp-1}, e_2, e_4, \dots, e_{2\nlow}]
\end{equation}
is a permutation matrix that is also orthogonal, and \(e_i\) denotes the \(i\)-th column of the \(n \times n\) identity matrix.
This ensures that the matrix \begin{equation}
    \bdR\trans \bdlda \bdR = \diag(\mu_1, \mu_3, \dots, \mu_{2\nupp-1}, \mu_2, \mu_4, \dots, \mu_{2\nlow})
\end{equation}
remains diagonal.
By straightforward computation, the orthogonal matrix \( \bdP \bdQ \bdR \) satisfies the two-block diagonalization property, which we denote as
\begin{equation}
\label{eq:PQR_participate_BH}
    \bdP \bdQ \bdR = \begin{bmatrix}
        \bdU_1 & 0 \\
             0 & \bdU_2
    \end{bmatrix},
\end{equation}
where \( \bdU_1 \in \R^{\nupp \times \nupp} \) and \( \bdU_2 \in \R^{\nlow \times \nlow} \) are orthogonal (for specific computing details, refer to Appendix).
Moreover, the diagonal matrix \( \bdR\trans \bdlda \bdR \) can also be partitioned into two diagonal blocks, whose sizes match those of \( \bdU_1 \) and \( \bdU_2 \), denoted by \( \bdlda_1 \) and \( \bdlda_2 \), where
\begin{equation}
\label{eq:lda1_lda2_BH}
    \Lambda_1 = \diag(\mu_1, \mu_3, \dots, \mu_{2\nupp-1}),\quad \Lambda_2 = \diag(\mu_2, \mu_4, \dots, \mu_{2\nlow}).
\end{equation}
Thus, \eqref{eq:yu_yl_BH} provides the eigendecompositions of both \( \yupp \) and \( \ylow \) simultaneously:
\begin{equation}
\label{eq:yu_yl_eigendecomposition_BH}
    \yupp = \bdU_1 \bdlda_1 \bdU_1\trans, \quad
    \ylow = \bdU_2 \bdlda_2 \bdU_2\trans.
\end{equation}

Using the eigendecomposition of \(\yupp\), the characteristic polynomial \(\fupp\) of \(\xupp\) can be derived:
\begin{equation}
\label{eq:xu_determinant_compute_BH}
\begin{aligned}
    \fupp(z) &= \det(zI - \xupp) \\
             &= \det(zI - \yupp - \vupp \,\wupp\trans) \\
             &= \det(zI - U_1 \Lambda_1 U_1\trans - \vupp \,\wupp\trans) \\
             &= \det(zI - \Lambda_1 - U_1\trans \vupp \,\wupp\trans U_1) \\
             &= \det(zI - \Lambda_1) 
                \cdot \det\Bigl(I - (zI -\Lambda_1)^{-1}U_1\trans \vupp \,\wupp\trans U_1\Bigr) \\
             &= \det(zI - \Lambda_1) 
                \cdot \Bigl(1 - \wupp \trans U_1(zI-\Lambda_1)^{-1}U_1\trans \vupp \Bigr).
\end{aligned}
\end{equation}
The second and third equalities rely on \eqref{eq:two_rank_one_BH} and \eqref{eq:yu_yl_eigendecomposition_BH}, respectively.
The last equality is derived using the Sylvester’s determinant identity (see Lemma~\ref{lemma:sylvester}).
The proof is finished by
\begin{equation}
    \begin{aligned}
        \wupp\trans U_1(zI-\Lambda_1)^{-1}U_1\trans \vupp &=-\frac{4}{n+1} \sum_{i=1}^{\nupp} \frac{\bigl(5\sin\gamma_i - 4\sin(2\gamma_i)  + \sin(3\gamma_i)\bigr)\sin\gamma_i}{z-\mu_{2i-1}} \\
        &=-\frac{64}{n+1} \sum_{i=1}^{\nupp} \frac{\sin^4(\gamma_i/2) \sin^2\gamma_i }{z - \mu_{2i-1}},
    \end{aligned}
\end{equation}
where $\gamma_i$ is denoted as $(2i-1)\pi/(n+1)$ for $1 \leq i \leq \nupp$.
\end{proof}

According to Lemma~\ref{lemma:charac_poly_BH}, the eigenvalues of the blocks $\xupp$ and $\xlow$ are precisely the zeros of the rational functions
\begin{equation}
\label{eq:rational_func}
    \gupp(z) = 1 + \sum_{i=1}^{\nupp} 
        \frac{\aupp_i}{z - \mu_{2i-1}},
    \quad 
    \glow(z) = 1 + \sum_{j=1}^{\nlow} 
        \frac{\alow_j}{z - \mu_{2j}},
\end{equation}
whose graphs are shown in Figure~\ref{fig:BH_rational}.
Figure~\ref{fig:BH_rational} shows that the zeros of $\gupp(z)$ lie respectively in the intervals $(0, \mu_1)$ and $(\mu_{2i-1}, \mu_{2i+1})$ for $1 \leq i \leq \nupp-1$.
Similarly, the zeros of $\glow(z)$ lie respectively in the intervals $(0, \mu_2)$ and $(\mu_{2j}, \mu_{2j+2})$ for $1 \leq j \leq \nlow-1$.
This property will be established rigorously in the next subsection.

\begin{figure}[!tb]
\centering
\begin{minipage}{0.48\linewidth}
    \centering
    \includegraphics[width=\linewidth]{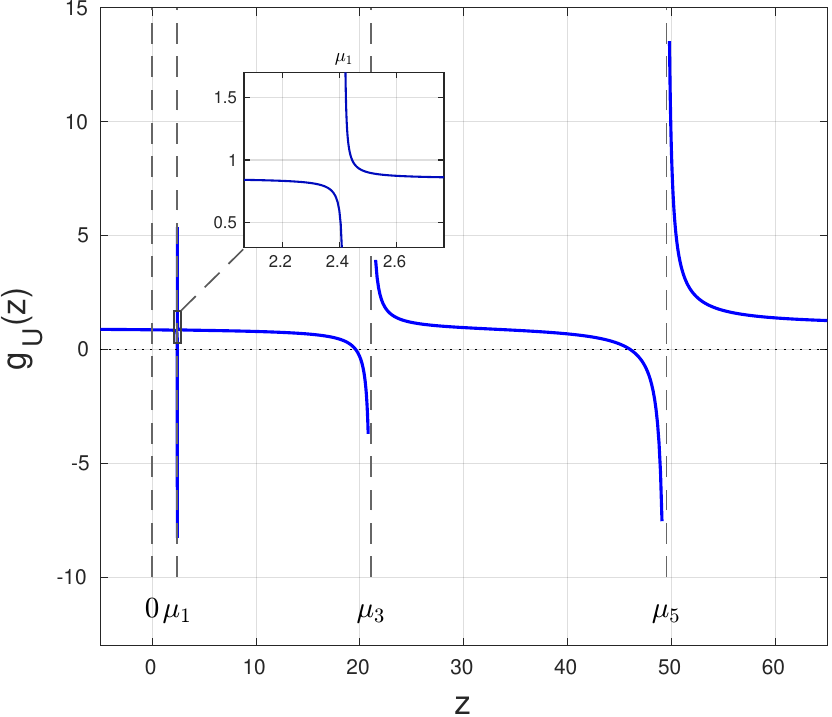}
    \subcaption{Upper rational function, $n = 6$.}
\end{minipage}
\hfill
\begin{minipage}{0.48\linewidth}
    \centering
    \includegraphics[width=\linewidth]{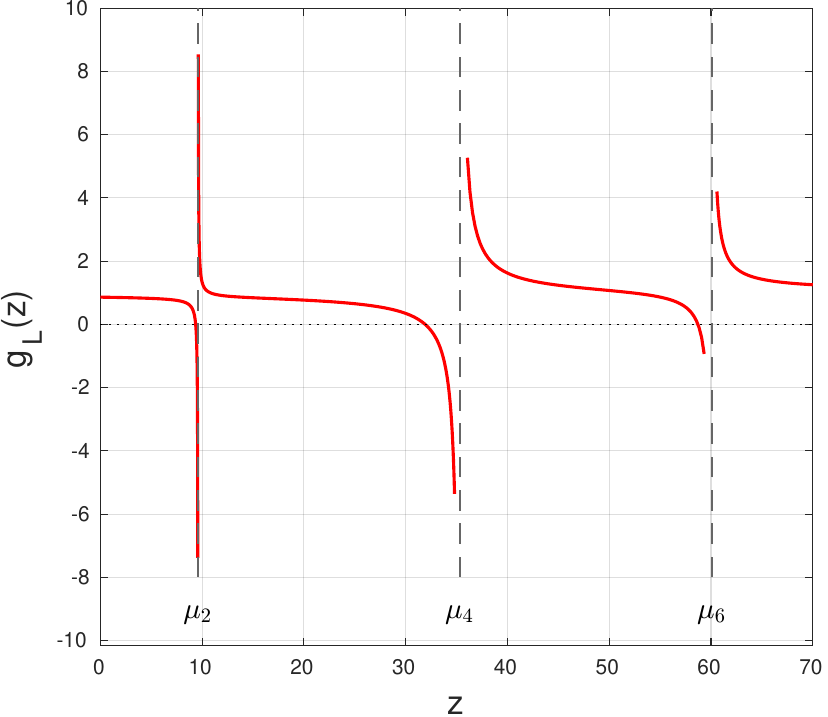}
    \subcaption{Lower rational function, $n = 6$.}
\end{minipage}
\caption{Graphs of the rational functions (a) $\gupp(z)$ and (b) $\glow(z)$ for $n = 6$.}
\label{fig:BH_rational}
\end{figure}

\subsection{\stepb: Spectral Analysis Based on the Characteristic Polynomial}
In this subsection, we rigorously analyze the distribution of the eigenvalues along the positive real axis by applying the intermediate value theorem to the characteristic polynomials of $\xupp$ and $\xlow$, by the Bramble--Hubbard scheme.
Main results are summarized in the following theorem.

\begin{theorem}
\label{thm:BH_fufl}
Let \(\bdL\) denote the discrete Laplacian \eqref{eq:discrete_Laplacian_matrix_BH} with matrix dimension \(n \ge 5\), constructed using the Bramble--Hubbard scheme
defined in~\eqref{eq:Bramble--Hubbard scheme}.\footnote{We note that the factor $1/(12h^2)$ has been omitted from \(\bdL\).}
Let \(\xupp\) and \(\xlow\) denote the corresponding upper and lower diagonal blocks in 
\eqref{eq:L^4_S_spectrum_decomposition_BH}. 
Then the following statements hold:
\begin{enumerate}
    \item The eigenvalues of \(\xupp\) are distributed over the intervals \((0,\, \mu_1)\)
    and \((\mu_{2i-1},\, \mu_{2i+1})\) for \(1 \le i \le \nupp-1\),
    with each interval containing exactly one eigenvalue, 
    where \(\mu_i\) is defined in~\eqref{eq:orthogonal_eigenmatrix_eigenvalue}.
    
    \item The eigenvalues of \(\xlow\) are distributed over the intervals \((0,\, \mu_2)\)
    and \((\mu_{2i},\, \mu_{2i+2})\) for \(1 \le i \le \nlow-1\),
    with each interval containing exactly one eigenvalue, 
    where \(\mu_i\) is defined in~\eqref{eq:orthogonal_eigenmatrix_eigenvalue}.
\end{enumerate}
\end{theorem}

\begin{proof}
We present the proof for the first part concerning the matrix~\(\xupp\); the proof for the second part, corresponding to~\(\xlow\), follows analogously.
We choose the grid points
\begin{equation}
\label{eq:xu_grid_point_BH}
    t_i = \mu_{2i-1}, \qquad 1 \le i \le \nupp,
\end{equation}
which form a strictly increasing, positive sequence. 
The value of the function $\fupp$ at each grid point $t_i$ is given by
\begin{equation}
\label{eq:xu_function_value_BH}
    \fupp(t_i) = \aupp_i \prod_{j \ne i} (t_i - t_j).
\end{equation}
Since the coefficients $\aupp_i$ defined in~\eqref{eq:bramble_hubbard_du} are strictly positive, 
\begin{equation}
   \sign(\fupp(t_i)) = \sign\Bigl(\,\prod_{j \ne i} (t_i - t_j)\,\Bigr) =  (-1)^{\nupp - i},
\end{equation}
where the sign function is defined as 
\begin{equation*}
    \sign(x) = \left\{ \begin{aligned}
        1&, & x > 0, \\
        0&, & x = 0, \\
        -1&, & x < 0.
    \end{aligned} \right.
\end{equation*}
It means that the sign of $f(t_i)$ alternates with the index $i$.
Thus, by the intermediate value theorem, \(\nupp - 1\) eigenvalues are located in the intervals 
\((t_i,\, t_{i+1})\) for \(1 \le i \le \nupp - 1\), with each interval containing at least one eigenvalue.
The remaining eigenvalue can then be determined by examining the interval \((0,\, t_1)\).

To prove there is one remaining root in the interval $(0, t_1)$, it is sufficient to prove that $\fupp(0)$ and $\fupp(t_1)$ have opposite signs.
By direct computation, we obtain
\begin{equation*}
\begin{aligned}
    \sign(\fupp(0) \fupp(t_1)) &= \sign(\fupp(0)) \sign(\fupp(t_1)) \\
        &= - \sign\Bigl(\,1 - \sum_{i=1}^{\nupp} 
    \frac{\aupp_i}{\mu_{2i-1}} \Bigr)
    \sign\Bigl(\,\prod_{i=1}^{\nupp}\mu_{2i-1}\Bigr) \\
        &= \sign\Bigl(\,\sum_{i=1}^{\nupp} 
    \frac{\aupp_i}{\mu_{2i-1}} - 1 \Bigr).
\end{aligned}
\end{equation*}
It is equivalent to proving that 
\begin{equation}
\label{eq:sufficient_cond_BH}
    1 - \sum_{i=1}^{\nupp} 
    \frac{\aupp_i}{\mu_{2i-1}} > 0.
\end{equation}
For convenience of notation, we denote $\gamma_i$ as $(2i-1)\pi/(n+1)$ for $1 \leq i \leq \nupp$. 
\begin{equation}
\begin{aligned}
    \sum_{i=1}^{\nupp} \frac{\aupp_i}{\mu_{2i-1}}
    &= \frac{64}{n+1} \sum_{i=1}^{\nupp}  \frac{\sin^4(\gamma_i/2)\, \sin^2\gamma_i}{28 - 32\cos\gamma_i + 4\cos^2\gamma_i} \\
    &= \frac{16}{n+1} \sum_{i=1}^{\nupp} \frac{\sin^4(\gamma_i/2)\, \sin^2\gamma_i}{(7 - \cos\gamma_i)(1 - \cos\gamma_i)} \\
    &= \frac{8}{n+1} \sum_{i=1}^{\nupp} \frac{\sin^2(\gamma_i/2)\,\sin^2\gamma_i}{7 - \cos\gamma_i} \\
    &< \frac{4}{3n+3} \sum_{i=1}^{\nupp} \sin^2\frac{\gamma_i}{2} \sin^2\gamma_i \\
    &= \frac{1}{6},
\end{aligned}
\end{equation}
where the last trigonometric identity is proved in Appendix~\ref{adx:sec:tri_identities}. 
Thus, we have established the inequality~\eqref{eq:sufficient_cond_BH}. 
Furthermore, the function values $\fupp(0)$ and $\fupp(t_1)$ have different signs. 
By applying the intermediate value theorem once again, we conclude that there are exactly $\nupp$ eigenvalues on the positive real axis.
Besides, these eigenvalues are distributed over the intervals \((0,\, t_1)\) and \((t_i,\, t_{i+1})\) for \(1 \le i \le \nupp - 1\), with each interval containing exactly one eigenvalue.
Since \(\deg(\fupp) = \nupp\), the proof is finished.
\end{proof}

\begin{corollary}
\label{coro:BH_fufl}
Let \(\bdL\) denote the discrete Laplacian \eqref{eq:discrete_Laplacian_matrix_BH} with matrix dimension \(n \ge 5\), constructed using the Bramble--Hubbard scheme
defined in~\eqref{eq:Bramble--Hubbard scheme}. 
Let \(\xupp\) and \(\xlow\) denote the corresponding upper and lower diagonal blocks in 
\eqref{eq:L^4_S_spectrum_decomposition_BH}.
Then all eigenvalues of both $\xupp$ and $\xlow$ are real and positive.
\end{corollary}

As a direct consequence of Corollary~\ref{coro:BH_fufl} and Lemma~\ref{lemma:charac_poly_BH}, 
we immediately obtain Theorem~\ref{thm:BH_real_positive}.

%% file: tex_src/spec_analysis_third_order_boundary.tex
\section{Spectral Analysis of Fourth-Order One-Sided Scheme}
\label{sec:spec_analysis_third_order}

In this section, we present our methodology for the spectral analysis of the discrete Laplacian from the fourth-order one-sided scheme \eqref{eq:Third-order near-boundary-accurate scheme},
\begin{equation}
\label{eq:discrete_Laplacian_matrix_third_order}
L = 
\begin{bmatrix}
20 & -6 & -4 & 1 \\
-16 & 30 & -16 & 1 \\
1 & -16 & 30 & -16 & 1 \\
& \ddots & \ddots & \ddots & \ddots & \ddots & \\
& & 1 & -16 & 30 & -16 & 1 \\
& & & 1 & -16 & 30 & -16 \\
& & & 1 & -4 & -6 & 20
\end{bmatrix},
\end{equation}
where the leading coefficient $1/(12h^2)$ is also omitted.
This simplification does not affect the validity of our results in this section.  

The proof framework follows exactly the same structure as that introduced in Section~\ref{sec:spec_analysis_BH}. 
Our approach consists of two main steps, referred to as \stepa\ and \stepb. 

\subsection{\stepa: Computation for the Characteristic Polynomial}

A direct computation of the characteristic polynomial of the discrete Laplacian $\bdL$ is not straightforward and does not fully exploit its underlying structure.
We therefore follow essentially the same procedure as for the Bramble--Hubbard scheme to compute the characteristic polynomial.
The main difference lies only in the specific numerical values, and we therefore omit the analogous details.


The discrete Laplacian \(\bdL\) generated by scheme \eqref{eq:Third-order near-boundary-accurate scheme} can be regarded as a rank-two perturbation of \(\bdS\),
\begin{equation}
\label{eq:rank_two_third_order}
    \bdL = \bdS + \bde_1 \bdz_1\trans + \bde_n \bdz_n\trans,
\end{equation}
where 
\begin{equation}
\label{eq:L-S_third_order}
    \bdz_1 = \bigl[ -9, 10, -5, 1, 0, \dots, 0\bigr]\trans, \quad
    \bdz_n = \bigl[0, \dots, 0, 1, -5, 10, -9 \bigr]\trans.
\end{equation}
Using the same approach as in Section~\ref{sec:spec_analysis_BH}, 
we apply Lemma~\ref{lemma:centro_block} to $L$, $S$, and $L - S$.
We derive that for \(n \geq 8\),
\begin{align}
\label{eq:L^4_S_spectrum_decomposition_third_order}
    & \bdP \bdL \bdP\trans = 
    \begin{bmatrix}
      \xupp & \bdzero \\
      \bdzero & \xlow
    \end{bmatrix},
    \quad
    \bdP \bdS \bdP\trans = 
    \begin{bmatrix}
      \yupp & \bdzero \\
      \bdzero & \ylow
    \end{bmatrix}, \\
\label{eq:L-S_spectrum_decomposition_third_order}
    & \hspace{4em} 
    \bdP (\bdL - \bdS) \bdP\trans = 
    \begin{bmatrix}
      \zupp & \bdzero \\
      \bdzero & \zlow
    \end{bmatrix}.
\end{align}
The submatrices $\xupp$, $\xlow$, $\yupp$, and $\ylow$ take slightly different forms depending on whether $n$ is even or odd.
The corresponding differences $\zupp$ and $\zlow$ share the same structure and can be expressed as
\begin{equation}
\label{eq:zupp_zlow_third_order}
    \zupp = \vupp \wupp\trans, \quad 
    \zlow = \vlow \wlow\trans,
\end{equation}
where the vectors $\vupp,\, \wupp \in \R^{\nupp}$ and $\vlow,\, \wlow \in \R^{\nlow}$ 
are given by
\begin{equation}
\label{eq:f1fn_third_order}
    \vupp = \bigl[1,\, 0,\, \dots,\, 0\bigr]\trans, 
    \quad 
    \vlow = \bigl[0,\, \dots,\, 0,\, 1\bigr]\trans,
\end{equation}
and
\begin{equation}
\label{eq:w1wn_third_order}
    \wupp = \bigl[-9,\, 10,\, -5,\, 1,\, 0,\, \dots,\, 0\bigr]\trans, 
    \quad
    \wlow = \bigl[0,\, \dots,\, 0,\, 1,\, -5,\, 10,\, -9\bigr]\trans.
\end{equation}
Thus, these two matrices are of rank one.
Substituting~\eqref{eq:zupp_zlow_third_order} into~\eqref{eq:L^4_S_spectrum_decomposition_third_order} yields
\begin{equation}
\label{eq:two_rank_one_third_order}
    \xupp = \yupp + \vupp \,\wupp\trans, 
    \quad 
    \xlow = \ylow + \vlow \,\wlow\trans.
\end{equation}

The rank-two correction \eqref{eq:rank_two_third_order} can thus be decomposed into two rank-one corrections~\eqref{eq:two_rank_one_third_order}.
This decomposition enables the characteristic polynomials of \(\xupp\) and \(\xlow\) to be computed separately 
by applying Sylvester’s determinant identity, 
together with the eigendecompositions of \(\yupp\) and \(\ylow\).

\begin{lemma}
\label{lemma:charac_poly_third_order}
Let \(\bdL\) denote the discrete Laplacian \eqref{eq:discrete_Laplacian_matrix_third_order} with matrix dimension $n \geq 5$, constructed by fourth-order one-sided scheme~\eqref{eq:Third-order near-boundary-accurate scheme}, and let \(\xupp\) and \(\xlow\) denote the corresponding upper and lower diagonal blocks
in~\eqref{eq:L^4_S_spectrum_decomposition_third_order}. 
Then the following results hold:
\begin{enumerate}
    \item The characteristic polynomial of \(\xupp\) can be expressed as 
    \begin{equation}
    \label{eq:fu_general_third_order}
        \fupp(z) 
        = \Bigl(\,1 + \sum_{i=1}^{\nupp} 
        \frac{\aupp_i}{z - \mu_{2i-1}} \Bigr)
        \prod_{i=1}^{\nupp}(z - \mu_{2i-1}),
    \end{equation}
    where \(\mu_{2i-1}\) is defined in~\eqref{eq:orthogonal_eigenmatrix_eigenvalue},
    \(\theta = \pi/(n + 1)\), and the coefficients \(\aupp\) take the following forms: for \(1 \le i \le \nupp\),
    \begin{equation}
    \label{eq:third_order_du}
        \aupp_i = 
            \frac{64}{n+1}
            \sin^4\!\frac{(2i-1)\pi}{2(n+1)}\,
            \sin^2\!\frac{(2i-1)\pi}{n+1}\,
            \Bigl(1 - 2\cos\!\frac{(2i-1)\pi}{n+1}\Bigr).
    \end{equation}
    
    \item The characteristic polynomial of \(\xlow\) can be expressed as 
    \begin{equation}
    \label{eq:fl_general_third_order}
        \flow(z) 
        = \Bigl(\,1 + \sum_{j=1}^{\nlow} 
        \frac{\alow_j}{z - \mu_{2j}} \Bigr)
        \prod_{j=1}^{\nlow}(z - \mu_{2j}),
    \end{equation}
    where \(\mu_{2j}\) is defined in~\eqref{eq:orthogonal_eigenmatrix_eigenvalue},
    \(\theta = \pi/(n + 1)\), and the coefficients \(\alow\) take the following forms: for \(1 \le j \le \nlow\),
    \begin{equation}
    \label{eq:third_order_dl}
            \alow_j = 
            \frac{64}{n+1}
            \sin^4\!\frac{j\pi}{n+1}
            \sin^2\!\frac{2j\pi}{n+1}\,
            \bigl(1 - 2\cos\!\frac{2j\pi}{n+1}\bigr).
    \end{equation}

    \item The characteristic polynomial of \(\bdL\) can be written as 
    \begin{equation*}
        f_{\Delta}(z) = \fupp(z)\,\flow(z).
    \end{equation*}
\end{enumerate}

\end{lemma}

\begin{proof}
For \(n \in \{5,6,7\}\), we can verify the  characteristic-polynomial identities directly by exact computation.
We therefore assume \(n \ge 8\) for the remainder of the proof.

We only provide the proof of part~1, as the proof of part~2 is analogous, and part~3 follows directly from~\eqref{eq:L^4_S_spectrum_decomposition_third_order}.
Using the same approach in the proof of Lemma~\ref{lemma:charac_poly_BH}, we can derive the eigendecomposition of $\yupp$ and $\ylow$,
\begin{equation}
\label{eq:yu_yl_eigendecomposition_third_order}
    \yupp = \bdU_1 \bdlda_1 \bdU_1\trans, \quad
    \ylow = \bdU_2 \bdlda_2 \bdU_2\trans,
\end{equation}
where matrices $U_1$, $U_2$, $\Lambda_1$ and $\Lambda_2$ are defined in \eqref{eq:PQR_participate_BH} and \eqref{eq:lda1_lda2_BH}.

Using the eigendecomposition of \(\yupp\), the characteristic polynomial \(\fupp\) of \(\xupp\) can be derived similarly as \eqref{eq:xu_determinant_compute_BH}:
\begin{equation}
\label{eq:xu_determinant_compute_third_order}
    \fupp(z) = \det(zI - \Lambda_1) 
                \cdot (1 - \wupp \trans U_1(zI-\Lambda_1)^{-1}U_1\trans \vupp).
\end{equation}
The proof is finished by
\begin{equation}
    \begin{aligned}
        & \hspace{1.5em} \wupp\trans U_1(zI-\Lambda_1)^{-1}U_1\trans \vupp \\
        &=-\frac{4}{n+1} \sum_{i=1}^{\nupp} \frac{\bigl(9\sin\gamma_i - 10\sin(2\gamma_i) + 5\sin(3\gamma_i) - \sin(4\gamma_i)\bigr)\sin\gamma_i}{z-\mu_{2i-1}} \\
        &=-\frac{64}{n+1} \sum_{i=1}^{\nupp} \frac{\sin^4(\gamma_i/2) \sin^2\gamma_i (1-2\cos\gamma_i)}{z-\mu_{2i-1}},
    \end{aligned}
\end{equation}
where $\gamma_i$ is denoted as $(2i-1)\pi/(n+1)$ for $1 \leq i \leq \nupp$.
\end{proof}

Different from the Bramble--Hubbard scheme, the relation between the zeros of rational functions
\begin{equation*}
    \gupp(z) = 1 + \sum_{i=1}^{\nupp} 
        \frac{\aupp_i}{z - \mu_{2i-1}},
    \quad 
    \glow(z) = 1 + \sum_{j=1}^{\nlow} 
        \frac{\alow_j}{z - \mu_{2j}},
\end{equation*}
and the eigenvalues of blocks $\xupp$ and $\xlow$ is more complicated.
This is because the coefficients $\aupp_i$ and $\alow_j$ are no longer always positive and may even vanish.
The graphs of the rational functions for $n=9$ are shown in Figure~\ref{fig:OS_rational}.
The zeros of $\gupp(z)$ and $\glow(z)$ are distributed over the intervals $(\mu_{2i-1}, \mu_{2i+1})$ for $1 \leq i \leq \nupp - 1$ and $(\mu_{2i}, \mu_{2i+2})$ for $1 \leq i \leq \nlow - 1$, respectively.
This property  will be established rigorously in the next subsection.

\begin{figure}[!tb]
\centering
\begin{minipage}{0.48\linewidth}
    \centering
    \includegraphics[width=\linewidth]{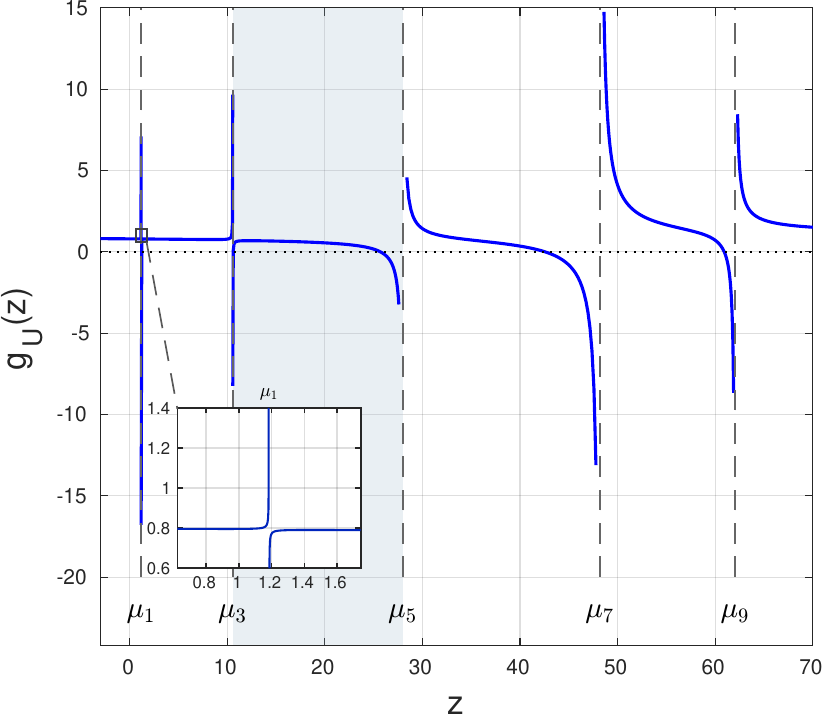}
    \subcaption{Upper rational function, $n = 9$.}
\end{minipage}
\hfill
\begin{minipage}{0.48\linewidth}
    \centering
    \includegraphics[width=\linewidth]{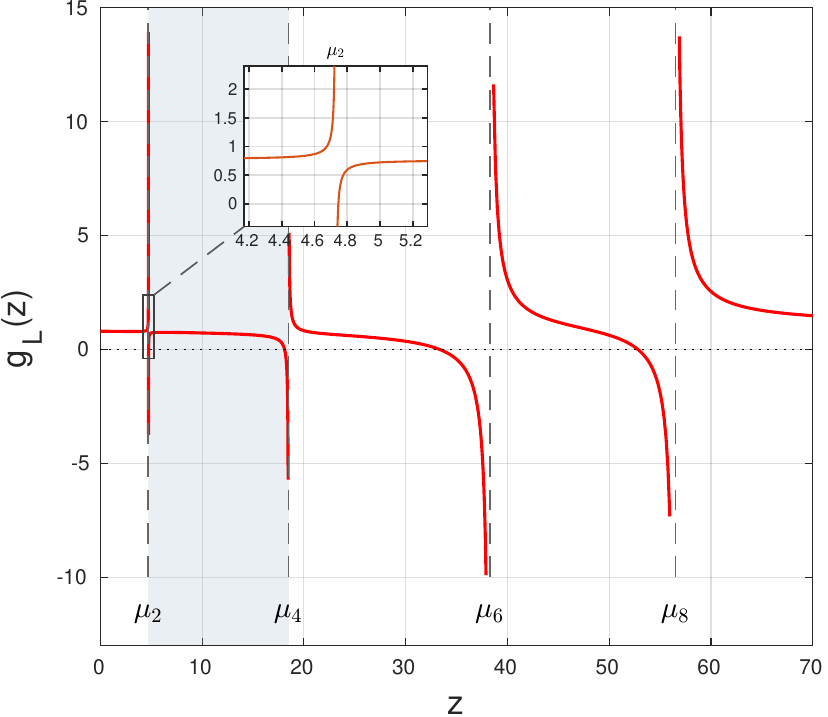}
    \subcaption{Lower rational function, $n = 9$.}
\end{minipage}
\caption{Graphs of the rational functions (a) $\gupp(z)$ and (b) $\glow(z)$ for $n = 9$.
The shaded region contains two zeros, whereas the other regions contain exactly one zero each. 
}
\label{fig:OS_rational}
\end{figure}

\subsection{\stepb: Spectral Analysis Based on the Characteristic Polynomial}
\label{subsec:stepb_third_order}
In this subsection, we rigorously analyze the distribution of the eigenvalues along the positive real axis by applying the intermediate value theorem.
Although this distribution is more complicated than that of the Bramble–Hubbard scheme, it does not introduce additional technical difficulties, but merely makes the analysis more involved.

\begin{theorem}
\label{thm:third_order_fufl}
Let \(\bdL\) denote the discrete Laplacian \eqref{eq:discrete_Laplacian_matrix_third_order} with matrix dimension \(n \ge 5\), constructed using the fourth-order one-sided scheme
defined in~\eqref{eq:Third-order near-boundary-accurate scheme}.\footnote{We note that the factor $1/(12h^2)$ has been omitted from \(\bdL\).}
Let \(\xupp\) and \(\xlow\) denote the corresponding upper and lower diagonal blocks in 
\eqref{eq:L^4_S_spectrum_decomposition_third_order}. 
Then all eigenvalues of both \(\xupp\) and \(\xlow\) are real and positive.
\end{theorem}

\begin{proof}
For \(n \in \{5,6,7\}\), exact Sturm sequence computations applied to the characteristic polynomials \(\fupp\) and \(\flow\) show that they have \(\nupp\) and \(\nlow\) distinct zeros in \((0,\infty)\), respectively.
We therefore assume \(n \ge 8\) for the remainder of the proof.

\paragraph{Proof for the case of the matrix \(\xupp\).}
We select the grid points as
\begin{equation}
\label{eq:xu_grid_point_third_order}
    t_i = \mu_{2i-1}, \qquad 1 \le i \le \nupp,
\end{equation}
which satisfy \(0<t_1<\cdots<t_{\nupp}\). 
The value of the function $\fupp$ at each grid point $t_i$ is given by
\begin{equation}
\label{eq:xu_function_value_third_order}
    \fupp(t_i) = \aupp_i \prod_{j \ne i} (t_i - t_j),
\end{equation}
whose signs alternate between positive and negative, except at a few special points. 

To examine this behavior, we decompose the expression into two parts: (i) the product term $\prod_{j \ne i} (t_i - t_j)$, and (ii) the coefficient $\aupp_i$.
For the first part, the sign of the product term alternates with respect to the index \(i\).  
For the second part, the non-positiveness of the coefficient $\aupp_i$ is equivalent to the condition
\begin{equation*}
    1 - 2\cos\!\frac{(2i-1)\pi}{n+1} \le 0
    \,\,\,\iff\,\,\, i \le  \Bigl\lfloor \frac{n + 4}{6} \Bigr\rfloor =: \mupp.
\end{equation*}
Moreover, this factor vanishes if and only if \(i = (n + 4)/6\) is an integer.  
Therefore, the specific sign pattern depends on whether 
\( n \not\equiv 2 \pmod{6} \) or \( n \equiv 2 \pmod{6} \).  
We summarize these results in Tables~\ref{tab:fu1} and~\ref{tab:fu2}.

\begin{table}[!tb]
\centering
\caption{
Sign pattern of $\fupp(t_i)$ in the case where \(n \not\equiv 2 \pmod{6}\).
The columns distinguish even and odd values of~\(\nupp\), while the rows distinguish the range and parity of~\(i\).
Here, “$+$” indicates a positive value, “$-$” a negative value, and “$0$” denotes zero.
}
\label{tab:fu1}
\begin{tabular}{cccc}
\toprule
\multicolumn{2}{c}{Index $i$} 
  & \multicolumn{2}{c}{Sign of $\fupp(t_i)$} \\
\cmidrule(lr){1-2}\cmidrule(lr){3-4}
Range of $i$ & Parity of $i$ 
  & $\nupp$ even & $\nupp$ odd \\
\midrule
\multirow{2}{*}{$i \le \mupp$}
  & odd  & $+$ & $-$ \\
  & even & $-$ & $+$ \\
\midrule
\multirow{2}{*}{$i > \mupp$}
  & odd  & $-$ & $+$ \\
  & even & $+$ & $-$ \\
\bottomrule
\end{tabular}
\end{table}

\begin{table}[!tb]
\centering
\caption{
Sign pattern of $\fupp(t_i)$ in the case where \(n \equiv 2 \pmod{6}\).
The columns correspond to even and odd values of~\(\nupp\), while the rows describe the range and parity of~\(i\).
Here, “$+$”, “$-$”, and “$0$” denote positive, negative, and zero values, respectively.
}
\label{tab:fu2}
\begin{tabular}{cccc}
\toprule
\multicolumn{2}{c}{Index $i$} 
  & \multicolumn{2}{c}{Sign of $\fupp(t_i)$} \\
\cmidrule(lr){1-2}\cmidrule(lr){3-4}
Range of $i$ & Parity of $i$ 
  & $\nupp$ even & $\nupp$ odd \\
\midrule
\multirow{2}{*}{$i < \mupp$}
  & odd  & $+$ & $-$ \\
  & even & $-$ & $+$ \\
\midrule
$i = \mupp$ & -- & $0$ & $0$ \\
\midrule
\multirow{2}{*}{$i > \mupp$}
  & odd  & $-$ & $+$ \\
  & even & $+$ & $-$ \\
\bottomrule
\end{tabular}
\end{table}

In the case where \( n \not\equiv 2 \pmod{6} \), by the intermediate value theorem, 
there exist \(\nupp - 2\) positive roots distributed across the intervals 
\(\left(t_i,\, t_{i+1}\right)\) for 
\(i \in \{1, 2, \ldots, \nupp - 1\} \setminus \{\mupp\}\), 
with each interval containing at least one positive root.

In contrast, when \( n \equiv 2 \pmod{6} \), the situation is slightly different: 
there are still \(\nupp - 2\) positive roots, but they are distributed across 
\(\nupp - 3\) intervals \(\left(t_i,\, t_{i+1}\right)\) for 
\(i \in \{1, 2, \ldots, \nupp - 1\} \setminus \{\mupp - 1,\, \mupp\}\), 
with each of these intervals containing at least one positive root, 
and one additional root located exactly at \(t_{\mupp}\).

In either case, the preceding sign analysis accounts for at least \(\nupp-2\) positive zeros, leaving two zeros to be located.
Figures~\ref{example:fu n=9} and~\ref{example:fu n=8} illustrate these two cases.

\begin{figure}[!t]
\begin{center}
\begin{tikzpicture}[x=1.5cm,y=1cm]
\draw[->] (-3,0) -- (3,0) node[right] {$x$};
\draw[->] (-2.8,-1.8) -- (-2.8,2) node[above] {$y$}; 
\filldraw[black] (-2.8,0) circle (1.5pt)
node[below left] {\footnotesize$O$};
\draw[thick] 
    (-2.4,-1.3)
    to[out=80, in=180] (-1.6,1)
    to[out=0, in=180] (-0.7,-0.5)
    to[out=0, in=180] (0.4,1)
    to[out=0, in=160] (1.4,-1)
    to[out=340, in=250] (2.4,1);

\draw[dashed] (-2.37,0) -- (-2.37,-0.92);
\draw[dashed] (-1.3,0) -- (-1.3, 0.75);
\draw[dashed] (-0.6,0) -- (-0.6,-0.48);
\draw[dashed] (0.05,0) -- (0.05,0.75);
\draw[dashed] (1.2,0) -- (1.2,-0.8);
\draw[dashed] (2.35,0) -- (2.35,0.8);

\filldraw[red] (-2.35,-0.92) circle (1.5pt) node[xshift=5pt,yshift=-15pt] {\footnotesize$\color{black}{(t_1,\fupp(t_1))}$};
\filldraw[red] (-1.3, 0.75) circle (1.5pt) node[above=10pt] {\footnotesize$\color{black}{(t_2,\fupp(t_2))}$};
\node[
  star,
  star points=5,
  star point ratio=2.25,
  fill=blue,
  draw=blue,
  inner sep=1.5pt,
  label={[below=15pt] {\footnotesize$\color{black}{\bigl(\tfrac{t_2+t_3}{2},\fupp(\tfrac{t_2+t_3}{2})\bigr)}$}}
] at (-0.6,-0.48) {};
\filldraw[red] (0.05,0.75) circle (1.5pt) node[xshift=5pt,yshift=18pt] {\footnotesize$\color{black}{(t_3,\fupp(t_3))}$};
\filldraw[red] (1.2,-0.8) circle (1.5pt) node[below=5pt] {\footnotesize$\color{black}{(t_4,\fupp(t_4))}$};
\filldraw[red] (2.35,0.8) circle (1.5pt) node[above=8pt] {\footnotesize$\color{black}{(t_5,\fupp(t_5))}$};

\draw[thick, fill=white] (-2.25,0) circle (2.8pt); 
\draw[thick] (-1.09+0.1,+0.1) -- (-1.09-0.1,-0.1); 
\draw[thick] (-1.09-0.1,0.1) -- (-1.09+0.1,-0.1); 
\draw[thick] (-0.24+0.1,0.1) -- (-0.24-0.1,-0.1); 
\draw[thick] (-0.24-0.1,0.1) -- (-0.24+0.1,-0.1); 
\draw[thick, fill=white] (0.93,0) circle (2.8pt);  
\draw[thick, fill=white] (2.17,0) circle (2.8pt);  

\end{tikzpicture}
\caption{Schematic sign diagram for \(n=9\), \(\nupp=5\), and \(\mupp=2\).
The figure illustrates the function \(\fupp\).
The circles `\(\circ\)' indicate zeros whose existence follows directly from the intermediate value theorem, whereas the crosses `\(\times\)' indicate potential zeros that cannot be identified directly by this argument.}
\label{example:fu n=9}
\end{center}
\end{figure}
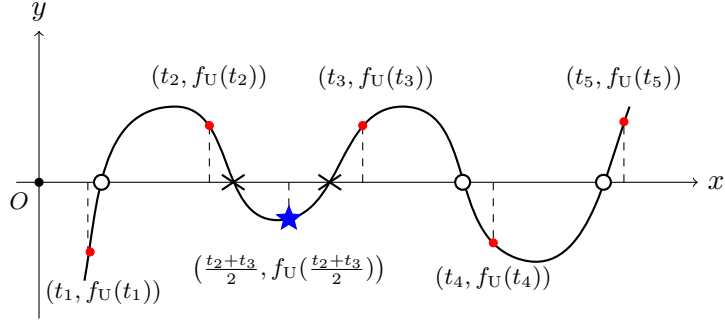

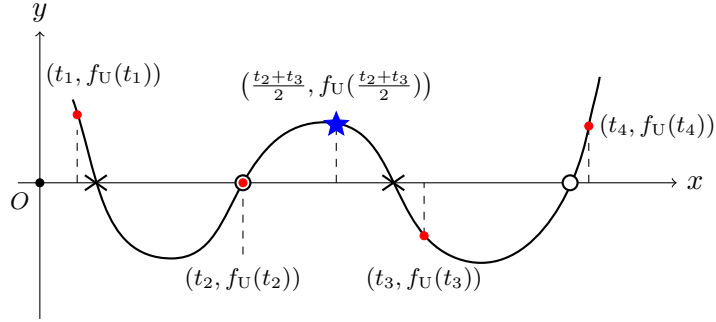
\begin{figure}[!t]
\begin{center}
\begin{tikzpicture}[x=1.45cm, y=1cm]
\draw[->] (-3,0) -- (3,0) node[right] {$x$};
\draw[->] (-2.8,-1.8) -- (-2.8,2) node[above] {$y$}; 
\filldraw[black] (-2.8,0) circle (1.5pt)
node[below left] {\footnotesize$O$};

\draw[thick] 
    (-2.5,1.1) 
    to[out=290, in=180] (-1.6,-1) 
    to[out=0, in=240] (-0.95,0) 
    to[out=60, in=180] (-0.2,0.8) 
    to[out=0, in=160] (1,-1) 
    to[out=340, in=260] (2.2,0.75)
    to[out=80, in=260] (2.3,1.4);

\draw[dashed] (-2.46,0) -- (-2.46,0.7);
\draw[dashed] (-0.95,0) -- (-0.95,-1);
\draw[dashed] (-0.1, 0) -- (-0.1, 0.77);
\draw[dashed] (0.7,-0.7) -- (0.7,0);
\draw[dashed] (2.2,0) -- (2.2,0.75);


\draw[thick] (-2.29+0.1,0.1) -- (-2.29-0.1,-0.1); 
\draw[thick] (-2.29-0.1,0.1) -- (-2.29+0.1,-0.1); 

\draw[thick, fill=white] (-0.95,0) circle (2.8pt); 

\draw[thick] (0.42+0.1,0.1) -- (0.42-0.1,-0.1); 
\draw[thick] (0.42-0.1,0.1) -- (0.42+0.1,-0.1); 

\draw[thick, fill=white] (2.03,0) circle (2.8pt);  

\filldraw[red] (-2.46,0.9) circle (1.5pt) node[xshift=10pt, yshift=15pt]{\footnotesize$\color{black}{(t_1,\fupp(t_1))}$};
\filldraw[red] (-0.95,0) circle (1.5pt) node[below=28pt] {\footnotesize$\color{black}{(t_2,\fupp(t_2))}$};
\node[
  star,
  star points=5,
  star point ratio=2.25,
  fill=blue,
  draw=blue,
  inner sep=1.5pt,
  label=above:{\footnotesize$\color{black}{\bigl(\tfrac{t_2 +t_3}{2},\fupp(\tfrac{t_2 +t_3}{2})\bigr)}$}
] at (-0.1, 0.77) {};
\filldraw[red] (0.7,-0.7) circle (1.5pt) node[below=8pt] {\footnotesize$\color{black}{(t_3,\fupp(t_3))}$};
\filldraw[red] (2.2,0.75) circle (1.5pt) node[right] {\footnotesize$\color{black}{(t_4,\fupp(t_4))}$};
\end{tikzpicture}
\caption{Schematic sign diagram for \(n=8\), \(\nupp=4\), and \(\mupp=2\).
The figure illustrates the function \(\fupp\).
The circles `\(\circ\)' indicate zeros whose existence follows directly from the intermediate value theorem, whereas the crosses `\(\times\)' indicate potential zeros that cannot be identified directly by this argument.}
\label{example:fu n=8}
\end{center}
\end{figure}

To identify the remaining two zeros, we examine the transition interval \((t_{\mupp},t_{\mupp+1})\).
Define the midpoint and the corresponding sum by
\begin{equation}
\label{eq:xu_midpoint_quantities_third_order}
    t_{\mathrm{mid}}
    =\frac{t_{\mupp}+t_{\mupp+1}}{2},
    \qquad
    \Phi_{\mathrm U}
    =\sum_{i=1}^{\nupp}
    \frac{2\aupp_i}{2t_i-t_{\mupp}-t_{\mupp+1}}.
\end{equation}
With the notation in~\eqref{eq:xu_midpoint_quantities_third_order}, Equation~\eqref{eq:fu_general_third_order} gives
\begin{equation}
\label{eq:xu_midpoint_factorization_third_order}
    \fupp(t_{\mathrm{mid}})
    =(1-\Phi_{\mathrm U})
    \prod_{i=1}^{\nupp}(t_{\mathrm{mid}}-t_i).
\end{equation}
To locate the remaining roots, we compare the midpoint value with the endpoint values of the transition interval.
By~\eqref{eq:xu_midpoint_factorization_third_order}, it suffices to prove \(\Phi_{\mathrm U}<1\), after which the midpoint sign follows directly.
For \(1\leq i\leq\nupp\), set \(\gamma_i=(2i-1)\pi/(n+1)=(2i-1)\theta\).
We first establish the local estimate
\begin{equation}
\label{eq:xu_local_estimate_third_order}
    \frac{1-2\cos\gamma_i}
    {2t_i-t_{\mupp}-t_{\mupp+1}}<\frac{1}{10},
    \qquad 1\leq i\leq\nupp.
\end{equation}
The proof of~\eqref{eq:xu_local_estimate_third_order} involves intricate but straightforward computations and is provided in Appendix~\ref{adx:sec:Details in the proof of Lemma}.
Combining~\eqref{eq:xu_local_estimate_third_order} with Lemma~\ref{lemma:tri_identities_OS} yields
\begin{equation}
\label{eq:xu_phi_bound_third_order}
\begin{aligned}
    \Phi_{\mathrm U}
    &<\frac{64}{5(n+1)}
      \sum_{i=1}^{\nupp}
      \sin^4\frac{\gamma_i}{2}\sin^2\gamma_i
    =1.
\end{aligned}
\end{equation}
By~\eqref{eq:xu_midpoint_factorization_third_order} and~\eqref{eq:xu_phi_bound_third_order}, the factor \(1-\Phi_{\mathrm U}\) is positive.
Since precisely \(\nupp-\mupp\) factors in \(\prod_{i=1}^{\nupp}(t_{\mathrm{mid}}-t_i)\) are negative,
\begin{equation}
\label{eq:xu_midpoint_sign_third_order}
    \sign\fupp(t_{\mathrm{mid}})=(-1)^{\nupp-\mupp}.
\end{equation}

If \(n\not\equiv2\pmod{6}\), Table~\ref{tab:fu1} and~\eqref{eq:xu_midpoint_sign_third_order} show that \(\fupp(t_{\mathrm{mid}})\) has the sign opposite to those of both \(\fupp(t_{\mupp})\) and \(\fupp(t_{\mupp+1})\).
Hence the intervals \((t_{\mupp},t_{\mathrm{mid}})\) and \((t_{\mathrm{mid}},t_{\mupp+1})\) each contain a root.
Together with the \(\nupp-2\) roots identified from the sign pattern in Table~\ref{tab:fu1}, these give at least \(\nupp\) positive roots.

It remains to consider \(n\equiv2\pmod{6}\).
Since \(n\geq8\), we have \(\mupp=(n+4)/6\geq2\), while \(\nupp-\mupp=(n-2)/3\) is even.
Table~\ref{tab:fu2} and~\eqref{eq:xu_midpoint_sign_third_order} give
\begin{equation}
\label{eq:xu_exceptional_signs_third_order}
    \fupp(t_{\mupp-1})>0,\qquad
    \fupp(t_{\mupp})=0,\qquad
    \fupp(t_{\mathrm{mid}})>0,\qquad
    \fupp(t_{\mupp+1})<0.
\end{equation}
The signs of \(\fupp(t_{\mathrm{mid}})\) and \(\fupp(t_{\mupp+1})\) in~\eqref{eq:xu_exceptional_signs_third_order} give a root in \((t_{\mathrm{mid}},t_{\mupp+1})\).
If \(\fupp\) is negative at some point of
\[
    (t_{\mupp-1},t_{\mupp})\cup(t_{\mupp},t_{\mathrm{mid}}),
\]
applying the intermediate value theorem between that point and \(t_{\mupp-1}\) in the first interval, or \(t_{\mathrm{mid}}\) in the second, supplies one further root distinct from \(t_{\mupp}\).
Otherwise, \(\fupp\geq0\) in a neighborhood of \(t_{\mupp}\).
Since \(\fupp(t_{\mupp})=0\), the point \(t_{\mupp}\) is a local minimum and \(\fupp'(t_{\mupp})=0\), so \(t_{\mupp}\) has algebraic multiplicity at least two.
Thus, in both the exceptional and nonexceptional cases, \(\fupp\) has at least \(\nupp\) positive zeros counted with algebraic multiplicity.
Since \(\deg(\fupp)=\nupp\), this proves the assertion for \(\xupp\).

\paragraph{Proof for the case of the matrix \(\xlow\).}
We now treat the lower block \(\xlow\), for which the sign transition is slightly different.
Define the grid points and the transition index by
\begin{equation}
\label{eq:xl_grid_data_third_order}
    s_j=\mu_{2j}\quad (1\leq j\leq\nlow),
    \qquad
    \mlow=\left\lfloor\frac{n+1}{6}\right\rfloor.
\end{equation}
For \(n\geq8\), the quantities in~\eqref{eq:xl_grid_data_third_order} satisfy \(1\leq\mlow<\nlow\), and \(0<s_1<\cdots<s_{\nlow}\).
The squared sine factors in~\eqref{eq:third_order_dl} are strictly positive, and hence
\begin{equation}
\label{eq:xl_coefficient_sign_third_order}
    \alow_j\leq0\iff j\leq\mlow,
    \qquad
    \alow_j=0
    \iff n\equiv5\!\!\!\pmod{6}\ \text{and}\ j=\mlow.
\end{equation}
At each grid point,
\begin{equation}
\label{eq:xl_grid_value_third_order}
    \flow(s_j)=\alow_j\prod_{k\ne j}(s_j-s_k).
\end{equation}
Combining~\eqref{eq:xl_coefficient_sign_third_order} and~\eqref{eq:xl_grid_value_third_order} gives, for every \(j\) such that \(\alow_j\ne0\),
\begin{equation}
\label{eq:xl_grid_sign_third_order}
    \sign\flow(s_j)=
    \begin{cases}
        (-1)^{\nlow-j+1}, & 1\leq j\leq\mlow,\\
        (-1)^{\nlow-j},   & \mlow<j\leq\nlow.
    \end{cases}
\end{equation}
If \(n\not\equiv5\pmod{6}\), Equation~\eqref{eq:xl_grid_sign_third_order} gives at least \(\nlow-2\) positive roots outside the transition interval.
If \(n\equiv5\pmod{6}\), Equations~\eqref{eq:xl_coefficient_sign_third_order}, \eqref{eq:xl_grid_value_third_order}, and~\eqref{eq:xl_grid_sign_third_order} give at least \(\nlow-3\) roots in the open grid intervals not adjacent to \(s_{\mlow}\), while \(s_{\mlow}\) is itself a root.
Thus, in either case, at least \(\nlow-2\) positive roots have already been accounted for, counting \(s_{\mlow}\) once in the latter case.

Define the midpoint and the corresponding sum by
\begin{equation}
\label{eq:xl_midpoint_quantities_third_order}
    s_{\mathrm{mid}}
    =\frac{s_{\mlow}+s_{\mlow+1}}{2},
    \qquad
    \Phi_{\mathrm L}
    =\sum_{j=1}^{\nlow}
    \frac{2\alow_j}{2s_j-s_{\mlow}-s_{\mlow+1}}.
\end{equation}
With the notation in~\eqref{eq:xl_midpoint_quantities_third_order}, Equation~\eqref{eq:fl_general_third_order} gives
\begin{equation}
\label{eq:xl_midpoint_factorization_third_order}
    \flow(s_{\mathrm{mid}})
    =(1-\Phi_{\mathrm L})
    \prod_{j=1}^{\nlow}(s_{\mathrm{mid}}-s_j).
\end{equation}
As in the upper-block case, by~\eqref{eq:xl_midpoint_factorization_third_order}, it suffices to prove \(\Phi_{\mathrm L}<1\), after which the midpoint sign follows directly.
For \(1\leq j\leq\nlow\), set \(\delta_j=2j\pi/(n+1)=2j\theta\).
The argument used to prove~\eqref{eq:xu_local_estimate_third_order} applies with \(\nupp,\mupp,t_i,\gamma_i\) replaced by \(\nlow,\mlow,s_j,\delta_j\), respectively, and yields
\begin{equation}
\label{eq:xl_local_estimate_third_order}
    \frac{1-2\cos\delta_j}
    {2s_j-s_{\mlow}-s_{\mlow+1}}<\frac{1}{10},
    \qquad 1\leq j\leq\nlow.
\end{equation}
Combining~\eqref{eq:xl_local_estimate_third_order} with Lemma~\ref{lemma:tri_identities_OS} gives
\begin{equation}
\label{eq:xl_phi_bound_third_order}
\begin{aligned}
    \Phi_{\mathrm L}
    &<\frac{64}{5(n+1)}
      \sum_{j=1}^{\nlow}
      \sin^4\frac{\delta_j}{2}\sin^2\delta_j
    =1.
\end{aligned}
\end{equation}
By~\eqref{eq:xl_midpoint_factorization_third_order} and~\eqref{eq:xl_phi_bound_third_order}, the factor \(1-\Phi_{\mathrm L}\) is positive.
Since precisely \(\nlow-\mlow\) factors in \(\prod_{j=1}^{\nlow}(s_{\mathrm{mid}}-s_j)\) are negative,
\begin{equation}
\label{eq:xl_midpoint_sign_third_order}
    \sign\flow(s_{\mathrm{mid}})=(-1)^{\nlow-\mlow}.
\end{equation}

If \(n\not\equiv5\pmod{6}\), Equations~\eqref{eq:xl_grid_sign_third_order} and~\eqref{eq:xl_midpoint_sign_third_order} show that \(\flow(s_{\mathrm{mid}})\) has the sign opposite to those of both \(\flow(s_{\mlow})\) and \(\flow(s_{\mlow+1})\).
Hence the intervals \((s_{\mlow},s_{\mathrm{mid}})\) and \((s_{\mathrm{mid}},s_{\mlow+1})\) each contain a root.
Together with the \(\nlow-2\) roots accounted for by~\eqref{eq:xl_grid_sign_third_order}, these give at least \(\nlow\) positive roots.

It remains to consider \(n\equiv5\pmod{6}\).
Since \(n\geq8\), we have \(\mlow=(n+1)/6\geq2\), while \(\nlow-\mlow=(n-2)/3\) is odd.
Equations~\eqref{eq:xl_coefficient_sign_third_order}, \eqref{eq:xl_grid_value_third_order}, \eqref{eq:xl_grid_sign_third_order}, and~\eqref{eq:xl_midpoint_sign_third_order} give
\begin{equation}
\label{eq:xl_exceptional_signs_third_order}
    \flow(s_{\mlow-1})<0,\qquad
    \flow(s_{\mlow})=0,\qquad
    \flow(s_{\mathrm{mid}})<0,\qquad
    \flow(s_{\mlow+1})>0.
\end{equation}
The signs of \(\flow(s_{\mathrm{mid}})\) and \(\flow(s_{\mlow+1})\) in~\eqref{eq:xl_exceptional_signs_third_order} give a root in \((s_{\mathrm{mid}},s_{\mlow+1})\).
If \(\flow\) is positive at some point of
\[
    (s_{\mlow-1},s_{\mlow})\cup(s_{\mlow},s_{\mathrm{mid}}),
\]
applying the intermediate value theorem between that point and \(s_{\mlow-1}\) in the first interval, or \(s_{\mathrm{mid}}\) in the second, supplies one further root distinct from \(s_{\mlow}\).
Otherwise, \(\flow\leq0\) in a neighborhood of \(s_{\mlow}\).
Since \(\flow(s_{\mlow})=0\), the point \(s_{\mlow}\) is a local maximum and \(\flow'(s_{\mlow})=0\), so \(s_{\mlow}\) has algebraic multiplicity at least two.
Thus, in both the exceptional and nonexceptional cases, \(\flow\) has at least \(\nlow\) positive zeros counted with algebraic multiplicity.
Since \(\deg(\flow)=\nlow\), this proves the assertion for \(\xlow\).
\end{proof}

As a direct consequence of Theorem~\ref{thm:third_order_fufl} and Lemma~\ref{lemma:charac_poly_third_order}, 
we immediately obtain Theorem~\ref{thm:third_order_real_positive}.
Higher-order extensions of the one-sided scheme need not preserve this property; selected counterexamples are described below.

\begin{remark*}
We demonstrate that several high-order one-sided schemes derived from~\eqref{eq:high_order_main_difference_scheme} can produce nonreal eigenvalues.
Providing a rigorous proof of the emergence of nonreal eigenvalues in general higher-order schemes is nontrivial.
However, for certain finite-difference orders and matrix dimensions, this spectral property can be analyzed using the symbolic Sturm sequence~\cite{basuAlgorithmsRealAlgebraic2006}.
For a fixed matrix dimension of $n=50$, Figure~\ref{fig:count_nonreal_eigenvalues} plots the number of nonreal eigenvalues against the stencil width $m=2\ell+1$.
All of the displayed widths with $m>5$ yield nonreal eigenvalues.
Figure~\ref{fig:sixth_order_spurious_eigenvalues} tracks these these numerically computed spurious eigenvalues as $n$ increases for the sixth-order scheme ($m=7$), showing their approximately quadratic growth away from the real axis.
\end{remark*}

\begin{figure}[!tb]
    \centering    \includegraphics[width=0.6\linewidth]{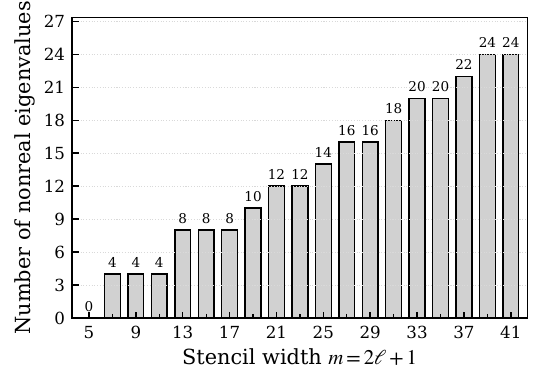}
    \caption{Relationship between the stencil width $m = 2\ell + 1$ and the number of nonreal eigenvalues for the discrete Laplacian associated with discretization~\eqref{eq:high_order_main_difference_scheme} with fixed dimension $n = 50$.}
    \label{fig:count_nonreal_eigenvalues}
\end{figure}

\begin{figure}[!tb]
    \centering
    \begin{minipage}{0.49\linewidth}
        \centering
        \includegraphics[width=\linewidth]{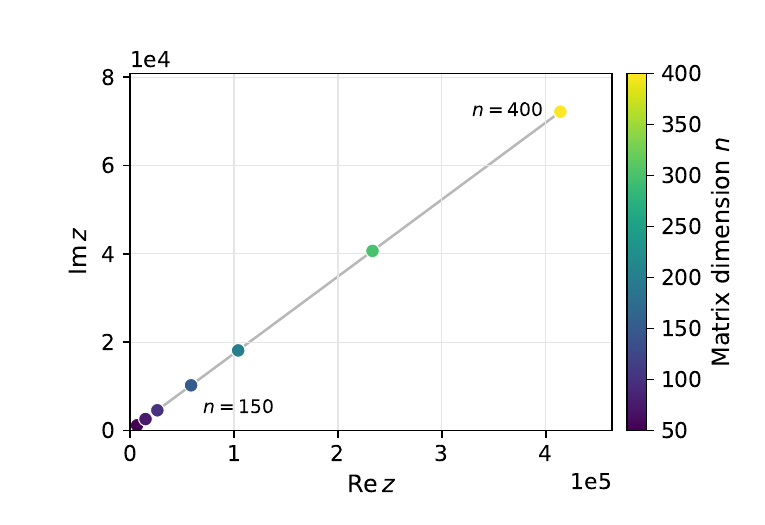}
    \end{minipage}%
    \hfill
    \begin{minipage}{0.49\linewidth}
        \centering
        \includegraphics[width=\linewidth]{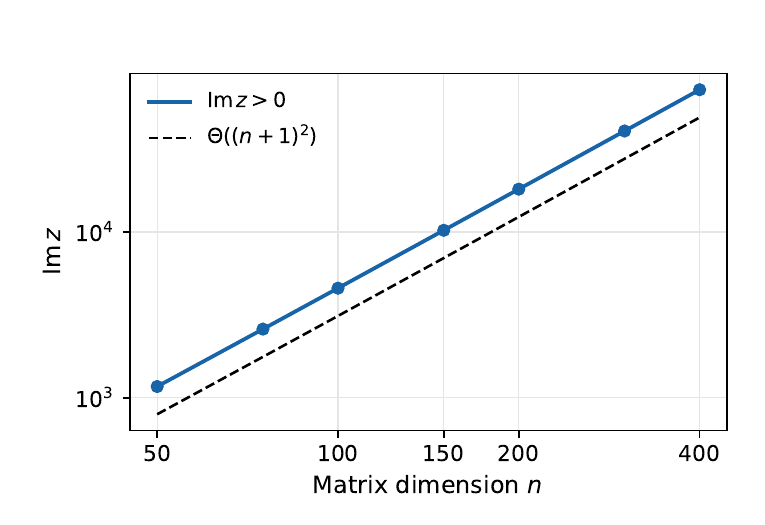}
    \end{minipage}
    \caption{Nonreal eigenvalues of the sixth-order one-sided scheme ($m=7$) for several matrix dimensions. The plotted eigenvalues are those of the original, fully scaled discrete Laplacian. Left: eigenvalues in the upper half-plane, colored by $n$. Right: their positive imaginary parts. Each matrix has two nonreal conjugate pairs, whose positive imaginary parts nearly coincide at this scale.}
    \label{fig:sixth_order_spurious_eigenvalues}
\end{figure}
\FloatBarrier

%% file: tex_src/conclusion.tex
\section{Conclusion}
\label{sec:conclusion}
This work establishes that two fourth-order finite-difference discretizations with Dirichlet boundary conditions, including the Bramble--Hubbard scheme and the one-sided scheme, produce nonsymmetric discrete Laplacians whose eigenvalues are real and strictly positive.
Computer-assisted Sturm sequence computations show that this property can fail for higher-order one-sided schemes.


%% file: tex_src/references.tex
\bibliographystyle{plainurl}
\bibliography{tex_src/reference}

%% file: tex_src/appendix.tex
\appendix\label{sec:appendix}
\section{On the Block Diagonalization Property}
\label{adx:sec:PQR_participate}
In this part, we will show the matrix $PQR$ admits the block-diagonalization, where matrices $P$, $Q$ and $R$ are defined in \eqref{eq:lemma_centro_block_even}, \eqref{eq:lemma_centro_block_odd}, \eqref{eq:orthogonal_eigenmatrix_eigenvalue} and \eqref{eq:permutation_matrix_R} respectively.
The whole computations are divided into two cases based on whether \(n\) is even or odd.
For simplicity for notations, we denote $\theta = \pi/(n+1)$.

\paragraph{Case 1: \( \boldsymbol{n} \) is even.}
We denote $r=n  /2$.
From \eqref{eq:orthogonal_eigenmatrix_eigenvalue}, the $i$-th column of $Q$ is expressed as 
\begin{equation*}
\label{eq:i_th_column_of_Q}
    Q_i =\sqrt{\frac{2}{n+1}}\, \bigl[\sin(i\theta),\sin(2i\theta),\dots, \sin(ni\theta)\bigr]\trans.
\end{equation*}
It leads to, for $1 \leq i \leq n$,
\begin{equation*}
\begin{aligned}
    PQ_i &= \frac{1}{\sqrt{n+1}}\Bigl[\sin(i\theta) + \sin(2ri\theta),\dots,\sin(ri\theta) + \sin((r+1)i\theta),\Bigr.
    \\
         &\hspace{5em}\Bigl.-\sin(ri\theta) + \sin((r+1)i\theta), \dots, -\sin(i\theta)+\sin(2ri\theta)\Bigr]\trans.
\end{aligned}
\end{equation*}
Furthermore, the matrix elements can be simplified by using the following trigonometric identities. 
For \(1 \le j \le r\) and \(1 \le i \le n\), we have
\begin{equation}
\label{eq:simplication_rules_1}
    \sin(ji\theta) + \sin((2r + 1 - j)i\theta)
    =
    \begin{cases}
        2\sin(ji\theta), & \text{if } 2 \nmid i, \\[3pt]
        0,                & \text{if } 2 \mid i,
    \end{cases}
\end{equation}
and
\begin{equation}
\label{eq:simplication_rules_2}
    -\sin(ji\theta) + \sin((2r + 1 - j)i\theta)
    =
    \begin{cases}
        0,                 & \text{if } 2 \nmid i, \\[3pt]
        -2\sin(ji\theta),  & \text{if } 2 \mid i.
    \end{cases}
\end{equation}
Applying these simplification rules to \(PQ R_i = PQ_{2i-1}\) for \(1 \le i \le r\), we obtain
\begin{equation*}
    PQR_i = \frac{2}{\sqrt{n+1}}\Bigl[\sin((2i-1)\theta),\, \sin(2(2i-1)\theta),\, \dots,\, \sin(r(2i-1)\theta),\, \underbrace{0, \dots,\, 0}_{r\text{ zeros}}\,\Bigr]\trans.
\end{equation*}
Similarly, for $1 \leq j \leq r$, the $(r+j)$-th column of $PQR$ can be simplified to
\begin{equation*}
    PQR_{r+j}=PQ_{2j}
    =-\frac{2}{\sqrt{n+1}}\bigl[\,
    \underbrace{0,\dots,0}_{r\text{ zeros}},\,
    \sin(2rj\theta),\,\sin(2(r-1)j\theta),\,\dots,\,\sin(2j\theta)
    \bigr]\trans.
\end{equation*}
Consequently, the matrix $PQR$ can be written as a block-diagonal form,
\begin{equation*}
    \bdP \bdQ \bdR = \begin{bmatrix}
        \bdU_1 & 0 \\
             0 & \bdU_2
    \end{bmatrix},
\end{equation*}
where the two submatrices are 
\begin{equation*}
\label{eq:U_1}
U_1 = \frac{2}{\sqrt{n+1}}\begin{bmatrix}
        \sin\theta & \sin(3\theta) & \dots & \sin((2r-1)\theta) \\
        \sin(2\theta) & \sin(6\theta) & \dots & \sin(2(2r-1)\theta) \\
        \vdots & \vdots & \ddots & \vdots \\
        \sin(r\theta) & \sin(3r\theta) & \dots & \sin(r(2r-1)\theta) 
    \end{bmatrix},
\end{equation*}
and
\begin{equation}
\label{eq:U_2}
U_2 = \frac{-2}{\sqrt{n+1}}\begin{bmatrix}
        \sin(2r\theta)       & \sin(4r\theta)       & \dots & \sin(2r^2\theta) \\
        \sin(2(r-1)\theta)   & \sin(4(r-1)\theta)   & \dots & \sin(2r(r-1)\theta) \\
        \vdots               & \vdots               & \ddots& \vdots \\
        \sin(2\theta)        & \sin(4\theta)        & \dots & \sin(2r\theta)
\end{bmatrix}.
\end{equation}

\paragraph{Case 2: \( \boldsymbol{n} \) is odd.}
In this case, an analogous procedure is applied as in the even-dimensional setting. 
We define \( r = (n - 1)/2 \). 
Through direct computation, we obtain the following expression for \( 1 \le i \le n \):
\begin{equation*}
\begin{aligned}
    PQ_i 
    &= \frac{1}{\sqrt{n + 1}}
    \Bigl[
        \sin(i\theta) + \sin((2r+1)i\theta),\, \dots,\, 
        \sin(ri\theta) + \sin((r + 2)i\theta),\, 
        \sqrt{2}\sin((r + 1)i\theta),\, \\
        &\hspace{4em}
        -\sin(ri\theta) + \sin((r + 2)i\theta),\, \dots,\,
        -\sin(i\theta) + \sin((2r+1)i\theta)
    \Bigr]^{\!\top}.
\end{aligned}
\end{equation*}
Following similar simplification rules as~\eqref{eq:simplication_rules_1} and \eqref{eq:simplication_rules_2} in the even case, we obtain
\begin{equation*}
U_1 = \frac{2}{\sqrt{n + 1}}
\begin{bmatrix}
    \sin\theta & \sin(3\theta) & \dots & \sin((2r + 1)\theta) \\
    \sin(2\theta) & \sin(6\theta) & \dots & \sin(2(2r + 1)\theta) \\
    \vdots & \vdots & \ddots & \vdots \\
    \sin(r\theta) & \sin(3r\theta) & \dots & \sin(r(2r + 1)\theta) \\
    \sin((r + 1)\theta)/\sqrt{2} & 
    \sin(3(r + 1)\theta)/\sqrt{2} & 
    \dots & 
    \sin((r + 1)(2r + 1)\theta)/\sqrt{2}
\end{bmatrix}.
\end{equation*}
and
\begin{equation*}
U_2 = \frac{-2}{\sqrt{n+1}}\begin{bmatrix}
        \sin(2r\theta)       & \sin(4r\theta)       & \dots & \sin(2r^2\theta) \\
        \sin(2(r-1)\theta)   & \sin(4(r-1)\theta)   & \dots & \sin(2r(r-1)\theta) \\
        \vdots               & \vdots               & \ddots& \vdots \\
        \sin(2\theta)        & \sin(4\theta)        & \dots & \sin(2r\theta)
\end{bmatrix}.
\end{equation*}

\section{On the Trigonometric Identities}
\label{adx:sec:tri_identities}

\begin{lemma}
\label{lemma:tri_sum}
For any $x \in \R$ and integer $n \geq 1$, the following identities hold:
\begin{equation}
\sin\frac{x}{2} \cdot \sum_{k=1}^n \sin(kx)
=
\sin\frac{nx}{2} \cdot \sin\frac{(n+1)x}{2},
\end{equation}
and
\begin{equation}
\sin\frac{x}{2} \cdot \sum_{k=1}^n \cos(kx)
=
\sin\frac{nx}{2} \cdot \cos\frac{(n+1)x}{2}.
\end{equation}
\end{lemma}



\begin{lemma}
\label{lemma:tri_identities_BH}
Let $n \geq 3$ be an integer.
Then the following identity hold:
\begin{equation*}
\begin{aligned}
\sum_{i=1}^{\lceil n/2\rceil}\sin^2\frac{(2i-1)\pi}{2n+2} \sin^2\frac{(2i-1)\pi}{n+1}=
\sum_{j=1}^{\lfloor n/2\rfloor}\sin^2\frac{j\pi}{n+1} \sin^2\frac{2\pi j}{n+1}=\frac{n+1}{8}.
\end{aligned}
\end{equation*}
\end{lemma}
\begin{proof}
Using the identities
\[
\sin^2\frac{x}{2}=\frac{1-\cos x}{2},
\qquad
\sin^2x=\frac{1-\cos 2x}{2},
\qquad
\cos x\cos 2x=\frac{\cos 3x+\cos x}{2},
\]
we obtain, for any $x\in\R$,
\begin{equation}
\label{eq:general_identity}
\sin^2\frac{x}{2}\,\sin^2x
=
\frac18\bigl(\cos(3x)-2\cos(2x)-\cos x+2\bigr).
\end{equation}

We first prove
\begin{equation}
\label{eq:tri_identity_BH_1}
\sum_{i=1}^{\lceil n/2\rceil}\sin^2\frac{(2i-1)\pi}{2n+2}\sin^2\frac{(2i-1)\pi}{n+1}
=
\frac{n+1}{8}.
\end{equation}
Set \(m=\lceil n/2 \rceil\) and \(\theta=\pi/(n+1)\).
Substituting $x=(2i-1)\theta$ into~\eqref{eq:general_identity} and summing over $i=1,\dots,m$, we obtain
\begin{equation}
\label{eq:first_sum_expand}
\sum_{i=1}^{m}\sin^2\frac{(2i-1)\theta}{2}\sin^2\bigl((2i-1)\theta\bigr)
=
\frac18\bigl(C_3-2C_2-C_1+2m\bigr),
\end{equation}
where
\[
C_r:=\sum_{i=1}^{m}\cos\bigl(r(2i-1)\theta\bigr), \qquad r=1,2,3.
\]
By Lemma~\ref{lemma:tri_sum},
\[
\sum_{i=1}^{m}\cos\bigl((2i-1)t\bigr)
=
\sum_{k=1}^{2m}\cos(kt)-\sum_{k=1}^{m}\cos(2kt)
=
\frac{\sin(2mt)}{2\sin t}.
\]
Hence we derive \(C_r=\sin(2mr\theta)/(2\sin(r\theta))\).
Now, since $m=\lceil n/2\rceil$, we have
\[
2m\theta=
\begin{cases}
n\theta=\pi-\theta, & n \text{ even},\\[2mm]
(n+1)\theta=\pi, & n \text{ odd}.
\end{cases}
\]
Therefore, for $n \geq 3$,
\[
C_1=
\begin{cases}
1/2, & n \text{ even},\\
0, & n \text{ odd},
\end{cases}
\qquad
C_2=
\begin{cases}
-1/2, & n \text{ even},\\
0, & n \text{ odd},
\end{cases}
\qquad
C_3=
\begin{cases}
1/2, & n \text{ even},\\
0, & n \text{ odd}.
\end{cases}
\]
Substituting these into~\eqref{eq:first_sum_expand}, in both cases we obtain~\eqref{eq:tri_identity_BH_1}.
For the second sum, the proof is similar.
\end{proof}

\begin{lemma}
\label{lemma:tri_identities_OS}
Let $n \geq 4$ be an integer.
Then the following identity hold:
\begin{equation*}
\begin{aligned}
\sum_{i=1}^{\lceil n/2\rceil}\sin^4\frac{(2i-1)\pi}{2n+2} \sin^2\frac{(2i-1)\pi}{n+1}=
\sum_{j=1}^{\lfloor n/2\rfloor}\sin^4\frac{j\pi}{n+1} \sin^2\frac{2\pi j}{n+1}=\frac{5n+5}{64}.
\end{aligned}
\end{equation*}
\end{lemma}
\begin{proof}
    Similar to Lemma~\ref{lemma:tri_identities_BH}.
\end{proof}

\section{On the Estimates Involving Trigonometric Functions}
\label{adx:sec:Details in 
the proof of Lemma}
In this part, we will prove the estimate: for \(1 \le i \le \nupp\):
\begin{equation}
\label{eq:xu_local_estimate_2_appendix}
    \frac{1 - 2\cos\gamma_i}{2t_i - t_{\mupp} - t_{\mupp+1}} 
    < \frac{1}{10},
\end{equation}
where $\gamma_i = (2i-1)\pi / (n+1)$, $\nupp = \lceil n/2 \rceil$, $\mupp = \lfloor (n+4)/6 \rfloor$ and $t_i$ is defined in~\eqref{eq:xu_grid_point_third_order}.
We observe that the denominator of the left-hand side term in \eqref{eq:xu_local_estimate_2_appendix} is negative for \(1 \leq i \leq \mupp\) and positive for \(\mupp < i \leq \nupp\).
By multiplying both sides by the denominator and expressing \(t_i\) as a quadratic polynomial in \(\cos\gamma_i\), the estimate \eqref{eq:xu_local_estimate_2_appendix} reduces to proving the quadratic polynomial
\begin{equation}
\label{eq:xu_quadratic_polynomial}
    g(z_i) = -8z_i^2 + 44z_i-46+t_{\mupp}+t_{\mupp+1}
\end{equation}
where \(z_i=\cos\gamma_i\), is positive for \(1 \leq i \leq \mupp\) and negative for \(\mupp < i \leq \nupp\).
By monotonicity, it suffices to check the boundary indices \(i=\mupp\) and \(i=\mupp+1\), namely \(g(z_{\mupp})>0\) and \(g(z_{\mupp+1})<0\).

Following the definition of the grid point \(t_i\) in \eqref{eq:xu_grid_point_third_order},
\begin{equation*}
    \begin{aligned}
    t_{\mupp} &= 4z_{\mupp}^2 - 32z_{\mupp} + 28 , \\
    t_{\mupp+1} &= 4z_{\mupp+1}^2 - 32z_{\mupp+1} + 28, \\
    \end{aligned}
\end{equation*}
substituting these into \eqref{eq:xu_quadratic_polynomial}, we obtain
\begin{equation*}
    \label{eq:g(z_mu_estimate)}
    \begin{aligned}
        g(z_{\mupp}) &= -4z_{\mupp}^2 + 4z_{\mupp+1}^2 + 12z_{\mupp} - 32z_{\mupp+1} + 10 \\
                     &=  (\underbrace{z_{\mupp} - z_{\mupp+1}}_{>0}) 
                     (\underbrace{12 - 4z_{\mupp} - 4z_{\mupp+1}}_{>0}) + 10(\underbrace{1-2z_{\mupp+1}}_{>0}) \\
                     & >0. 
    \end{aligned}
\end{equation*}
Similarly,
\begin{equation*}
    \begin{aligned}
        g(z_{\mupp+1}) &= -4z_{\mupp+1}^2 + 4z_{\mupp}^2 + 12z_{\mupp+1} - 32z_{\mupp} + 10 \\
                     &=  (\underbrace{z_{\mupp+1} - z_{\mupp}}_{<0}) 
                     (\underbrace{12 - 4z_{\mupp} - 4z_{\mupp+1}}_{>0}) + 10(\underbrace{1-2z_{\mupp}}_{\leq 0}) \\
                     & <0.
                     \\  
    \end{aligned}
\end{equation*}
This completes the proof of estimate \eqref{eq:xu_local_estimate_2_appendix}. 
For \(n\geq8\), the same argument, with \(\nupp,\mupp,t_i,\gamma_i\) replaced by \(\nlow,\mlow,s_j,\delta_j\), respectively, proves~\eqref{eq:xl_local_estimate_third_order}.

\section{\texorpdfstring{Computer-Assisted Exact Verifications}%
{Computer-Assisted Exact Verifications}}
\label{adx:sec:computer_assisted_exact_verification}
The exact arithmetic Python scripts supporting the computer-assisted
statements in this paper are available in \href{https://github.com/aaduiduidui/fourth-order-laplacian-code}{https://github.com/aaduiduidui/fourth-order-laplacian-code}.
They use exact integer, rational, or algebraic number arithmetic; no
floating point eigenvalue computation is used.
The repository provides
reproducible verifications of the following statements.
\begin{enumerate}
    \item For \(n\in\{5,6,7\}\), the upper  and lower block
    characteristic polynomial formulas in Lemma~\ref{lemma:charac_poly_BH}
    are checked coefficient by coefficient in exact algebraic number fields.

    \item For \(n\in\{5,6,7\}\), the upper and lower block
    characteristic polynomial formulas in
    Lemma~\ref{lemma:charac_poly_third_order} are checked coefficient by
    coefficient in exact algebraic number fields.

    \item For the fourth-order one-sided scheme with
    \(n\in\{5,6,7\}\), exact Sturm sequence computations verify that the
    upper and lower characteristic polynomials have \(\nupp\) and \(\nlow\)
    distinct roots in \((0,\infty)\), respectively, thereby verifying the
    corresponding small dimensional cases of
    Theorem~\ref{thm:third_order_fufl}.

    \item For \(n=50\) and \(m=5,7,\ldots,41\), exact
    characteristic polynomials and Sturm sequence root counts reproduce the
    numbers of nonreal eigenvalues reported in
    Figure~\ref{fig:count_nonreal_eigenvalues}.
\end{enumerate}